\documentclass[12pt]{amsart}
\usepackage{amsmath,amsthm,amsfonts,latexsym,amssymb,amscd,xcolor}
\usepackage{chemarr}
\usepackage{mathtools}
\usepackage{mathrsfs}
\usepackage{tikz}
\numberwithin{equation}{section}
\newtheorem{thm}[equation]{Theorem} 
\newtheorem{cor}[equation]{Corollary}
\newtheorem{lm}[equation]{Lemma}

\newtheorem{clm}[equation]{Claim}
\newtheorem{goal}[equation]{Goal}
\newtheorem{subclm}[equation]{Subclaim}
\newtheorem*{clm*}{Claim}
\theoremstyle{definition}
\newtheorem{df}[equation]{Definition}

\newtheorem{example}[equation]{Example}

\newcommand{\cproof}{\noindent{\it Proof of Claim.}\ } 
\newcommand{\cqed}{\hfill\rule{1.3mm}{3mm}}

\newcommand{\wec}[1]{{\mathbf{#1}}}  
\newcommand{\wrel}[1]{\;#1\;}     

\newcommand{\m}[1]{{\mathbf{\uppercase{#1}}}}

\DeclareMathOperator{\Con}{Con}

\newcommand{\cg}{\mathrm{Cg}}
\newcommand{\C}[1]{{\mathbf{\uppercase{#1}}}}
\newcommand{\Cg}[1]{{\cg}^{\m{#1}}}

\newcommand{\solv}{\stackrel{s}{\sim}}

\newcommand{\lhdlhd}{\mathrel{\lhd\!\!\lhd}}

\begin{document}

\title[]{Relative Maltsev definability \\ of some commutator properties}

\author{Keith A. Kearnes}
\address[Keith A. Kearnes]{Department of Mathematics\\
University of Colorado\\
Boulder, CO 80309-0395\\
USA}
\email{kearnes@colorado.edu}


\begin{abstract}
  We show that, when restricted to the class of varieties
  that have a Taylor term, several commutator properties
  are definable by Maltsev conditions.
\end{abstract}

\maketitle

\section{Introduction}\label{intro}
A {\it strong Maltsev condition} is a positive primitive sentence
in the language of clones. 
That is, it is a sentence expressing the existence of
some clone elements satisfying some equalities.
The name derives from the original example in \cite{maltsev}:
A.~I.~Maltsev proved that the class of varieties
whose members have permuting congruences 
is exactly the class of varieties whose
clones satisfy the p.p.\ clone sentence
\begin{equation} \label{sigma}
\tag*{$\sigma:$}
(\exists p)((p(x,x,y)\approx y)\;\&\; (p(x,y,y)\approx x)).
\end{equation}
A variety is said to satisfy a strong
Maltsev condition if its clone does.
In this article I will say that a class of varieties
is {\it definable by a
  strong Maltsev condition} if it is exactly the class
of all varieties that satisfy the strong Maltsev condition.

A {\it Maltsev condition} (without the word \emph{strong})
is a sequence
$\Sigma = (\sigma_n)_{n\in\omega}$ of successively weaker
strong Maltsev conditions ($\sigma_n\vdash \sigma_{n+1}$ for all $n$).
A variety $\mathcal V$ satisfies $\Sigma$ if its clone
satisfies $\sigma_n$ for some $n$. A class of varieties
is {\it definable by a Maltsev condition},
or is {\it Maltsev definable},
if it is exactly the class
of all varieties that satisfy some Maltsev condition $\Sigma$.
A class of varieties definable by a strong Maltsev condition $\sigma$
is also definable by the ordinary Maltsev condition
$\Sigma = (\sigma, \sigma, \ldots)$ that is a constant sequence.

In this article,
I investigate classes of varieties that are not Maltsev definable,
but which become Maltsev definable relative to some weak
`ground' Maltsev condition. That is,
suppose that $\mathscr{P}$ is a property of varieties.
Let $\Gamma$ be a Maltsev condition.
I will investigate some instances where
the class of varieties satisfying $\mathscr{P}$
is not Maltsev definable, but the class of varieties
satisfying both $\mathscr{P}$ and $\Gamma$ is Maltsev definable.
In symbols, I might write ${\mathscr P}+\Gamma = \Sigma$
to mean that, restricted to varieties satisfying
the ground condition $\Gamma$, the
class of varieties satisfying the condition ${\mathscr P}$
is definable by the Maltsev condition $\Sigma$.
I then say that the class of varieties
satisfying $\mathscr{P}$ is {\it Maltsev definable
relative to $\Gamma$}.
In this article, the `ground' Maltsev condition
will always be `the existence of a Taylor term'.
It is known, through Corollaries~5.2 and 5.3 of
\cite{taylor}, that an idempotent variety has a
Taylor term if and only if it contains no algebra
with at least $2$ elements in which every operation
interprets as a projection operation.
It is easy to see that this means exactly that
`the existence of a Taylor term' is the weakest
nontrivial idempotent Maltsev condition.
It is known that `the existence of a Taylor term'
is expressible as a strong Maltsev condition, see \cite{olsak}.

We investigate relative Maltsev definability for
the ten commutator properties $\mathscr{P}$
from the following list.
To understand these statements completely, it is necessary
to know the definitions of
$\C C(x,y;z)$ (Definition~\ref{centralizer_def}),
of $[x,y]$ (Definition~\ref{commutator_def}), 
and of (relative) right or left annihilators
(Definition~\ref{annihilator_def}).
For intuition about these statements, it may help
to remember that for the variety of groups ``the commutator operation''
coincides with the usual commutator operation of group theory
($[M,N]=[M,N]_{\textrm{group}}$)
while for the variety of commutative rings ``the commutator operation''
coincides with ideal product ($[I,J]=I\cdot J$).
``The centralizer relation'', $\C C(x,y;z)$,
coincides with the relation $[x,y]\leq z$ in both cases.
For commutative rings,
``the relative (right or left) annihilator of $J$ modulo $I$''
is the colon ideal $(I:J)=\{r\in R\;|\;rJ\subseteq I\}$, while ``the annihilator
of $J$'' is special case $(0:J)$.

\begin{itemize}
\setlength{\itemindent}{-10pt}
\item $[x,y]=[y,x]$
(Commutativity of the commutator.) 
\item $[x+y,z]=[x,z]+[y,z]$ (Left distributivity of the commutator.) 
\item $[x,y+z]=[x,y]+[x,z]$ (Right distributivity of the commutator.) 
\item $[x,y]=[x,y']\Longrightarrow [x,y]=[x,y+y']$
  (Right semidistributivity of the commutator.) 
\item Given $x$, there exists a largest $y$ such that
$[x,y]=0$
(Right annihilators exist.)
\item Given $x, z$, there exists a largest $y$ such that
$\C C(x,y;z)$
  (Relative right annihilators exist.) We write $(z:x)_R$
  for the relative right annihilator of $x$ modulo $z$,
  when it exists.
\item $\C C(x,y;z)\Longleftrightarrow \C C(y,x;z)$
(Symmetry of the centralizer relation in its first two places.) 
\item $\C C(x,y;z)\Longleftrightarrow [x,y]\leq z$
(The centralizer relation is determined by the commutator.)
\item $\C C(x,y;z)\;\&\; (z\leq z')\Longrightarrow \C C(x,y;z')$
(Stability of the centralizer relation under lifting in its third place.)
\item $\C C(x,y;z)\;\&\; (z\leq z'\leq x\cap y)\Longrightarrow \C C(x,y;z')$
(Weak stability of the centralizer relation under lifting in its third place.)
\end{itemize}

The main results of this article may be summarized as follows.
First, I explain why no one of the ten commutator properties
listed above is Maltsev definable [Section~\ref{examples}].
Then I explain why the following are equivalent
for varieties $\mathcal V$ with a Taylor term:
\begin{itemize}
\setlength{\itemindent}{-10pt}  
\item $\mathcal V$ is congruence modular.
\item The commutator is left distributive throughout $\mathcal V$.
\item The commutator is right distributive throughout $\mathcal V$.
\item The centralizer relation is symmetric in its first two places throughout $\mathcal V$.
\item Relative right annihilators exist throughout $\mathcal V$.  
\item The centralizer relation is determined by the commutator.  
\item The centralizer relation is stable under lifting in its third place.
\end{itemize}
See Theorems~\ref{main2} and \ref{main3}.
Also, for varieties $\mathcal V$ with a Taylor term,
the following are equivalent:
\begin{itemize}
\setlength{\itemindent}{-10pt}  
\item $\mathcal V$ has a difference term.
\item The commutator is commutative throughout $\mathcal V$.
\item Right annihilators exist throughout $\mathcal V$.
\item The commutator is right semidistributive throughout $\mathcal V$.
\item The centralizer relation is weakly stable under lifting in its third place.  
\end{itemize}
See Theorems~\ref{main1}, \ref{main1.5}, \ref{main4}.

A specific Maltsev
condition defining the class of congruence modular varieties 
may be found in \cite[Section~2]{day}.
A specific Maltsev
condition defining the class of varieties with a difference term
may be found in \cite[Section~4]{kissterm}.
Thus, Theorems~\ref{main1}, \ref{main2}, \ref{main1.5}, 
\ref{main3}, and \ref{main4} establish the 
Maltsev definability of all ten commutator
properties relative to the existence of a Taylor term.

The proofs of relative Maltsev definability for the
ten commutator properties identified will be called
the ``primary'' results of this article, and will be identified
as such when we prove them.
All other results are considered ``secondary'',
although some secondary results are as interesting
as the primary results. For example,
some nontrivial commutator-theoretic
facts are proved in Section~\ref{facts}
whose proofs do not require the
existence of a Taylor term.
In addition to this, 
a commutator-theoretic characterization
of the class of varieties that have a weak difference term
is established
in Theorem~\ref{characterization_of_weak}.

For background, I direct the reader to Section~2.4 of \cite{shape}
for a discussion of Maltsev conditions and
Section~2.5 of \cite{shape} for a discussion
of the properties of the centralizer relation
$\C C(x,y;z)$.
The most important elements from this
source will be reproduced below when needed.
In particular, it will be necessary
to know the definitions of
$\C C(x,y;z)$ (Definition~\ref{centralizer_def}),
of $[x,y]$ (Definition~\ref{commutator_def}), 
of (relative) right or left annihilators
(Definition~\ref{annihilator_def}),
of a difference term (Definition~\ref{left_right}),
and of a Taylor term (see
the opening paragraph of Section~\ref{main}).
\bigskip

\section{The ten commutator
  properties are not Maltsev definable}\label{examples}

The variety $\mathcal V$ of sets has the properties that
$\C C(\alpha,\beta;\delta)$ and $[\alpha,\beta]=0$ hold
for any $\alpha,\beta,\delta\in \Con(\m a)$,
$\m a\in {\mathcal V}$.
This implies that each of the following are true
in the variety of sets:

\begin{itemize}
\item $[x,y]=[y,x]$ 
\item $[x+y,z]=[x,z]+[y,z]$ 
\item $[x,y+z]=[x,y]+[x,z]$ 
\item $[x,y]=[x,y']\Longrightarrow [x,y]=[x,y+y']$
\item Given $x$, there exists a largest $y$ such that
$[x,y]=0$
\item Given $x, z$, there exists a largest $y$ such that
$\C C(x,y;z)$
\item $\C C(x,y;z)\Longleftrightarrow \C C(y,x;z)$ 
\item $\C C(x,y;z)\Longleftrightarrow [x,y]\leq z$
\item $\C C(x,y;z)\;\&\; (z\leq z')\Longrightarrow \C C(x,y;z')$
\item $\C C(x,y;z)\;\&\; (z\leq z'\leq x\cap y)\Longrightarrow \C C(x,y;z')$  
\end{itemize}

\noindent
If one of these properties $\mathscr P$
were Maltsev definable,
then, since the variety of sets is interpretable in any variety,
every variety would satisfy $\mathscr P$.
To prove that no one of these properties
is Maltsev definable it suffices
to exhibit varieties where the properties fail.

All the properties fail in the variety of semigroups
$\mathcal V =
{\sf H}{\sf S}{\sf P}(\mathbb Z_2 \times\mathbb S_2)$
where $\mathbb Z_2$ is the $2$-element group considered
as a semigroup and $\mathbb S_2$ is the $2$-element semilattice.
This is a variety of commutative semigroups
satisfying $x^3\approx x$.
In this variety, the term $T(x,y,z) = xyz$
is a Taylor term for $\mathcal V$
(see \cite[Definition~2.15]{shape} or
the opening paragraph of Section~\ref{main} below).
One can conclude this by noting that
$T$ is idempotent in $\mathcal V$
(since $T(x,x,x)\approx x^3\approx x$)
and satisfies $i$-th place Taylor identities
in $\mathcal V$ for every $i$
(since $T(x,y,z)\approx xyz\approx zxy\approx T(z,x,y)$).

From the main results of this article,
the fact that $\mathcal V$ has a Taylor term implies
that, if $\mathcal V$ had one of the commutator
properties listed above, then
$\mathcal V$ would have a difference term.
Then, from Theorem~\ref{diff_char} below,
any pentagon in a congruence lattice of a member
of $\mathcal V$ would have a `neutral' critical interval.
This is not the case, since $\Con(\mathbb Z_2\times \mathbb S_2)$
is a pentagon and its critical interval is abelian.
The lattice $\Con(\mathbb Z_2\times \mathbb S_2)$
is indicated in Figure~\ref{fig1}
with some congruences identified using the notation
``partition : congruence'' or 
``congruence : partition''.
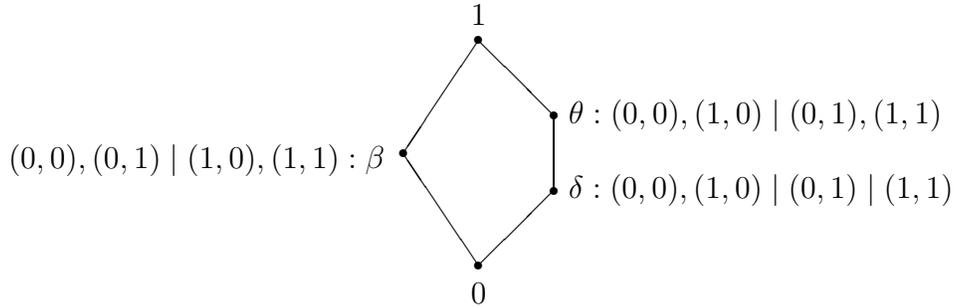
\begin{figure}[ht]
\begin{center}
\setlength{\unitlength}{1mm}
\begin{picture}(20,33)
\put(0,15){\circle*{1.2}}
\put(10,0){\circle*{1.2}}
\put(10,30){\circle*{1.2}}
\put(20,10){\circle*{1.2}}
\put(20,20){\circle*{1.2}}

\put(10,0){\line(-2,3){10}}
\put(10,0){\line(1,1){10}}
\put(10,30){\line(-2,-3){10}}
\put(10,30){\line(1,-1){10}}
\put(20,20){\line(0,-1){10}}

\put(-52.5,13){$(0,0), (0,1)\mid (1,0), (1,1) :\beta$}
\put(22,9){$\delta : (0,0), (1,0)\mid (0,1) \mid (1,1)$}
\put(22,19){$\theta: (0,0), (1,0)\mid (0,1), (1,1)$}
\put(9,32){$1$}
\put(9,-5){$0$}
\end{picture}
\medskip

\caption{\sc $\Con(\mathbb Z_2\times \mathbb S_2)\cong \m n_5$.}\label{fig1}
\end{center}
\end{figure}
\noindent
By hand computations, or by UACalc \cite{uacalc}, it can be shown that
$\C C(\theta,\theta;\delta)$. This means that the critical
interval of this copy of $\m n_5$ is abelian, hence
is \emph{not} neutral.

\section{Commutator theoretic results true for every variety}\label{facts}

In this section we prove some new facts about the commutator 
which we will need later in the paper.
They have been extracted from their rightful
places in the next section
and recorded here solely because the proofs
require no ground Maltsev condition among their hypotheses.

Our notation follows that of \cite{freese-mckenzie},
and we direct the reader to that source for fuller explanations.
For example, $\Con(\m a)$ is the congruence
lattice of $\m a$.
The meet (= intersection) and join 
of congruences $\alpha, \beta\in\Con(\m a)$
will be denoted $\alpha\cap \beta$ and $\alpha+\beta$.
We might write $u\stackrel{\alpha}{\equiv}v$
as an alternative to $(u,v)\in\alpha$.
When $\delta\leq \theta$, $I[\delta,\theta]$ denotes
the interval in $\Con(\m a)$ consisting of all
congruences between $\delta$ and $\theta$,
namely $I[\delta,\theta]=\{x\in\Con(\m a)\;|\;\delta\leq x\leq \theta\}$.
A five-element
sublattice of $\Con(\m a)$ is called a pentagon
if it is isomorphic to the lattice depicted in Figure~\ref{fig1}.
The critical interval of a pentagon is the interval that
corresponds to $I[\delta,\theta]$ in Figure~\ref{fig1}.
If $\alpha\in\Con(\m a)$, then
$\m a(\alpha)$ denotes the subalgebra of $\m a\times \m a$
whose universe is $\alpha$
(reference for notation: page 37 of \cite{freese-mckenzie}).
We use product notation for congruences of $\m a(\alpha)$,
so for $\beta,\gamma\in\Con(\m a)$ we let
$\beta_1 = \{((x,y),(z,w))\in A(\alpha)^2\;|\;(x,z)\in\beta\}$
  and
  $\gamma_2 = \{((x,y),(z,w))\in A(\alpha)^2\;|\;(y,w)\in\gamma\}$
  (reference: page 85 of \cite{freese-mckenzie}).\footnote{Observe
  that the definitions of the relations
  $\beta_1$ and $\gamma_2$ depend on the choice of $\alpha$.}
  Following \cite{freese-mckenzie}, we deviate from
  this convention by using $\eta_1$ and $\eta_2$
  in place of $0_1$ and $0_2$. E.g.,
$\eta_1 = \{((x,y),(z,w))\in A(\alpha)^2\;|\;x=z\}$.
We typically write $\beta_1\times \gamma_2$ for $\beta_1\cap \gamma_2$.
Given $\beta\in\Con(\m a)$, we let $\Delta_{\alpha,\beta}$ be the congruence
on $\m a(\alpha)$ generated by the $\beta$-diagonal relation
$
\{((x,x),(z,z))\in A(\alpha)^2\;|\;(x,z)\in\beta\}
$
(reference: page 37 of \cite[Definition~4.7]{freese-mckenzie}).
A fact we use when necessary is that
\begin{equation} \label{delta_gen}
  \Delta_{\alpha,\beta}\leq \beta_1\times \beta_2
\end{equation}  
always holds, since the generators of the congruence
$\Delta_{\alpha,\beta}$ lie in the congruence $\beta_1\times \beta_2$.

Next we define $S,T$-matrices and 
the centralizer relation.
The definitions are made for \emph{tolerances}
of an algebra $\m a$.
A tolerance on $\m a$ is a reflexive, symmetric, compatible
binary relation. (A \emph{congruence}
is a transitive tolerance.)

\begin{df}\label{matrices_def}
If $S$ and $T$ are tolerances on an algebra $\m a$, then an 
{\bf $S,T$-matrix}
is a $2\times 2$
matrix of elements of $\m a$ of the form 
\[
\left[\begin{array}{cc}
p&q\\
r&s
\end{array}\right]=
\left[\begin{array}{cc}
t(\wec{a},\wec{u})&t(\wec{a},\wec{v})\\
t(\wec{b},\wec{u})&t(\wec{b},\wec{v})
\end{array}\right]\]
where $t(\wec{x},\wec{y})$ is an $(m+n)$-ary term operation
of $\m a$, $\wec{a} \wrel{S} \wec{b}$, and
$\wec{u}\wrel{T}\wec{v}$.
The set of all $S,T$-matrices
of $\m a$ is denoted $M(S,T)$.
\end{df}

The symmetry of tolerances guarantees
that the set $M(S,T)$ is invariant under the operations
of interchanging rows or columns.

\begin{df}\label{centralizer_def}
  Let $S$ and $T$ be tolerances of an algebra
  $\m a$ and let $\delta$ be a congruence on $\m a$.
If $p\stackrel{\delta}{\equiv} q$ implies that $r\stackrel{\delta}{\equiv} s$
whenever
$$
\left[\begin{array}{cc}
p&q\\
r&s
\end{array}\right]\in M(S,T),
$$
then we say that 
{\bf $\C C(S,T;\delta)$ holds},
or {\bf $S$ centralizes $T$ modulo $\delta$}.
\end{df}

Many of the basic properties of the centralizer
relation are proved in Theorem~2.19 of \cite{shape}.
I copy the statement of that theorem here because its many
items will be referenced repeatedly throughout
this article.

\begin{thm}\label{basic_centrality}
Let $\m a$ be an algebra with tolerances 
$S, S', T, T'$
and congruences 
$\alpha, \alpha_i,\beta,\delta,\delta',\delta_j$.
The following are true.
\begin{enumerate}
\item[(1)] {\rm (Monotonicity in the first two variables)}
If $\C C(S,T;\delta)$ holds 
and $S'\subseteq S$,
$T'\subseteq T$, then $\C C(S',T';\delta)$ holds.
\item [(2)] 
$\C C(S,T;\delta)$ holds if and only if 
$\C C(\Cg a(S),T;\delta)$ holds.
\item [(3)] $\C C(S, T; \delta)$ holds if and only if 
$\C C(S, \delta\circ T\circ \delta; \delta)$ holds.
\item [(4)] If $T\cap \delta = T\cap\delta'$, then
$\C C(S,T;\delta)$ holds if and only if 
$\C C(S,T;\delta')$ holds.
\item [(5)] {\rm (Semidistributivity in the first 
variable)}
If $\C C(\alpha_i,T;\delta)$ holds
for all $i\in I$, then
$\C C(\bigvee_{i\in I}\alpha_i,T;\delta)$ holds.
\item[(6)] If $\C C(S,T;\delta_j)$ holds 
for all $j\in J$, then
$\C C(S,T;\bigwedge_{j\in J}\delta_j)$ holds.
\item [(7)] If $T\cap(S\circ (T\cap\delta)\circ S)
\subseteq\delta$, then $\C C(S,T;\delta)$ holds.
\item [(8)] 
If $\beta\cap(\alpha+(\beta\cap\delta))\leq\delta$,
then $\C C(\alpha,\beta;\delta)$ holds.
\item [(9)] Let $\m b$ be a subalgebra of $\m a$.
If $\C C(S, T; \delta)$ holds in $\m a$, 
then
$\C C(S|_{\m b}, T|_{\m b}; \delta|_{\m b})$ holds in $\m b$.
\item [(10)] 
If $\delta'\leq \delta$, then the relation
$\C C(S, T; \delta)$ holds in $\m a$ if and only if
$\C C(S/\delta', T/\delta'; \delta/\delta')$ 
holds in $\m a/\delta'$.
\end{enumerate}
\end{thm}

The commutator operation is defined in terms of the centralizer
relation.

\begin{df}\label{commutator_def}
Let $S,T$ be tolerances on
an algebra $\m a$.
The commutator $[S,T]$ equals the least congruence $\delta$ on $\m a$
for which $\C C(S,T;\delta)$ holds.
\end{df}

According to Definition~\ref{commutator_def}, $[S,T]=0$ holds if
and only if $\C C(S,T;0)$ holds.
By Theorem~\ref{basic_centrality}~(5),
if $T$ is a tolerance on some algebra, 
then the join $\alpha$ of all congruences $\alpha_i$
satisfying $\C C(\alpha_i,T;0)$
satisfies $\C C(\alpha,T;0)$.
Using Theorem~\ref{basic_centrality}~(2)
we see that this join $\alpha$ is a congruence
and it is the largest congruence $x$ such that
$\C C(x,T;0)$ or equivalently the largest
$x$ such that $[x,T]=0$.
We denote this largest $x$ by $(0:T)$
and call it the annihilator of $T$.
If we want to emphasize that the annihilator $x=(0:T)$ appears
in the left variable of the commutator
in the equation $[x,T]=0$ we will add
a subscript $L$ to write $x=(0:T)_L$ 
and say that $(0:T)_L$ the left annihilator of $T$.
For the same reasons, given a tolerance
$T$ and a congruence $\delta\in\Con(\m a)$
there exists a largest tolerance $\alpha$ such that
$\C C(\alpha,T;\delta)$ which we denote $(\delta:T)$ or
$(\delta:T)_L$.
We call $(\delta:T)_L$ the \underline{relative} left annihilator
of $T$ \underline{modulo $\delta$}.
We record
the notation we have just introduced
in Definition~\ref{annihilator_def}.
Although the definitions from \cite{shape}
of the centralizer relation and the commutator operation
involve tolerance relations (reflexive, symmetric, compatible
binary relations) rather than congruence relations
(transitive tolerances), in this paper we henceforth
concentrate on the centralizer, commutator, and annihilators
of congruences only.

\begin{df}\label{annihilator_def}
Let $\m a$ be an algebra and let $\delta, \beta\in\Con(\m a)$
be congruences on $\m a$.
\begin{enumerate}
\item 
The largest congruence $\alpha\in\Con(\m a)$
satisfying $\C C(\alpha,\beta;0)$ is called 
the \emph{left annihilator of $\beta$} and it is
denoted $(0:\beta)_L$.
If there is a  largest congruence $\alpha\in\Con(\m a)$
satisfying $\C C(\beta,\alpha;0)$ it is called 
the \emph{right annihilator of $\beta$} and it is
denoted $(0:\beta)_R$.
\item  The \emph{relative left annihilator of $\beta$ modulo $\delta$},
  denoted $(\delta:\beta)_L$,
  is the largest congruence $\alpha\in\Con(\m a)$
  satisfying $\C C(\alpha,\beta;\delta)$.
If there is a  largest congruence $\alpha\in\Con(\m a)$
satisfying $\C C(\beta,\alpha;\delta)$ it is called 
the \emph{relative right annihilator of $\beta$ modulo $\delta$}, and it is
denoted $(\delta:\beta)_R$.
\end{enumerate}
\end{df}

\bigskip

The first new result in this section
shows that if a variety contains
an algebra
whose congruence lattice contains
a certain kind of
pentagon with an abelian critical interval,
then the variety contains an algebra
with
a pentagon satisfying other (usually stronger)
abelianness conditions.

\begin{figure}[ht]
\begin{center}
\setlength{\unitlength}{1mm}
\begin{picture}(20,33)
\put(0,15){\circle*{1.2}}
\put(10,0){\circle*{1.2}}
\put(10,30){\circle*{1.2}}
\put(20,10){\circle*{1.2}}
\put(20,20){\circle*{1.2}}

\put(10,0){\line(-2,3){10}}
\put(10,0){\line(1,1){10}}
\put(10,30){\line(-2,-3){10}}
\put(10,30){\line(1,-1){10}}
\put(20,20){\line(0,-1){10}}

\put(-4.5,13){$\beta$}
\put(22,9){$\delta$}
\put(22,19){$\theta$}
\put(9,32){$\alpha$}
\put(9,-5){$0$}
\end{picture}
\bigskip

\caption{\sc $\textrm{Con}(\m A)$ or $\textrm{Con}(\m B)$.}\label{fig2}
\end{center}
\end{figure}
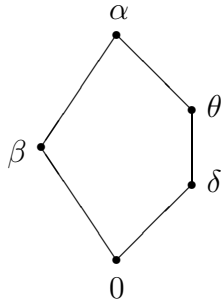

\begin{thm} \label{better_pentagons}
{\rm  (Better pentagons)}
  Let $\mathcal V$ be an arbitrary
  variety and assume that $\mathcal V$ contains an algebra $\m a$
  with congruences $\beta, \theta,\delta$
  generating a pentagon, as shown in Figure~\ref{fig2}.
  Assume that $\C C(\theta,\theta;\delta)$ holds
  and $\C C(\beta,\theta;\delta)$ fails.\footnote{The assumption that
 ``$\C C(\beta,\theta;\delta)$ fails'' will always hold if
$\mathcal V$ has a Taylor term -- see Theorem~\ref{memoir_pentagon}.}  
There exists an algebra $\m b\in \mathcal V$
with congruences ordered as in Figure~\ref{fig2}
and satisfying $\C C(\alpha,\alpha;\beta)$
and $\C C(\theta,\theta;0)$.
\end{thm}  

\begin{proof}
  Let $\m a$ have congruences $\beta, \theta, \delta$
  with the properties described.
  We will find a pentagon of the desired type
  in the congruence lattice of the algebra $\m a(\beta)$.

  The desired pentagon will be the one depicted in
  Figure~\ref{fig3}.
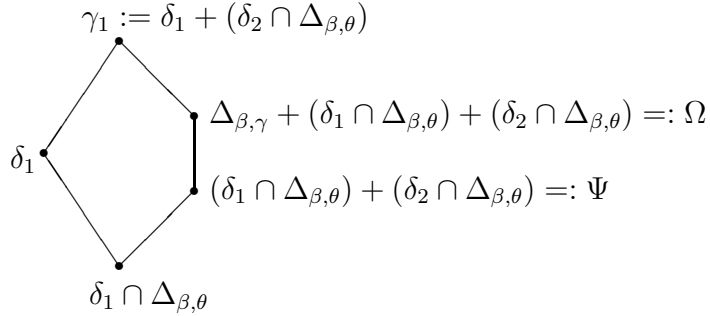
\begin{figure}[ht]
\begin{center}
\setlength{\unitlength}{1mm}

\begin{picture}(20,38)
\put(-20,0){%
\put(0,15){\circle*{1.2}}
\put(10,0){\circle*{1.2}}
\put(10,30){\circle*{1.2}}
\put(20,10){\circle*{1.2}}
\put(20,20){\circle*{1.2}}

\put(10,0){\line(-2,3){10}}
\put(10,0){\line(1,1){10}}
\put(10,30){\line(-2,-3){10}}
\put(10,30){\line(1,-1){10}}
\put(20,20){\line(0,-1){10}}

\put(-4.5,13){$\delta_1$}
\put(22,9){$(\delta_1\cap \Delta_{\beta,\theta})+(\delta_2\cap\Delta_{\beta,\theta})=:\Psi$}
\put(22,19){$\Delta_{\beta,\gamma}+(\delta_1\cap \Delta_{\beta,\theta})+(\delta_2\cap\Delta_{\beta,\theta})=:\Omega$}
\put(5,32){$\gamma_1:=\delta_1+(\delta_2\cap\Delta_{\beta,\theta})$}
\put(6,-5){$\delta_1\cap\Delta_{\beta,\theta}$}
}
\end{picture}
\bigskip

\caption{\sc A sublattice of $\textrm{Con}(\m A(\beta))$.}\label{fig3}
\end{center}
\end{figure}
The congruence $\gamma$ that appears in this
figure will be defined in the course of the proof.

To establish that the congruences
indicated form a pentagon with the required properties
we will use the fact that $\C C(\beta,\theta;\delta)$
fails. According to Definition~\ref{centralizer_def},
this assumption yields 
a $\beta,\theta$-matrix
\[
\begin{bmatrix}p&q\\r&s
\end{bmatrix}=
\begin{bmatrix}
t(\wec{a},\wec{c})&t(\wec{a},\wec{d})\\
t(\wec{b},\wec{c})&t(\wec{b},\wec{d})
\end{bmatrix}
\]
where $(a_i,b_i)\in\beta$, $(c_j,d_j)\in\theta$,
$(p,q)\in\delta$, and $(r,s)\in\theta-\delta$.
This implies that
$((r,p), (s,q))\in \Delta_{\beta,\theta}\cap \delta_2$,
$((r,p), (s,q))\in \Delta_{\beta,\theta}\cap \theta_1$,
but 
$((r,p), (s,q))\notin \delta_1$.

The comparabilities indicated in the chain
on the left side
in Figure~\ref{fig3}, namely
\begin{equation} \label{chain1}
\delta_1\cap\Delta_{\beta,\theta}\;\leq\; \delta_1\;\leq\; \delta_1+(\delta_2\cap\Delta_{\beta,\theta}),
\end{equation}
are obvious:
the middle congruence
$\delta_1$ is a meetand on the left of \eqref{chain1}
and a joinand
on the right of \eqref{chain1}.
To show that the middlemost and the rightmost elements of
chain \eqref{chain1} are distinct, observe that 
$((r,p), (s,q))$ lies in the
rightmost congruence of \eqref{chain1}
but not in the middlemost.
To show that the middlemost and leftmost elements of
chain \eqref{chain1} are distinct,
choose $(u,v)\in \beta-\theta$.
This is possible since $\beta\not\leq\theta$ (see
Figure~\ref{fig2}).
Both pairs $(u,u), (u,v)$ belong to $\m a(\beta)$
and $((u,u), (u,v))\in\delta_1$.
We have $\Delta_{\beta,\theta}\leq \theta_1\times \theta_2$
by \eqref{delta_gen}, 
so $\delta_1\cap \Delta_{\beta,\theta}\leq \delta_1\times \theta_2$.
However the pair 
$((u,u), (u,v))$ does not belong to
$\delta_1\times \theta_2$, since $(u,v)\notin\theta$.
This shows that
$((u,u), (u,v))$ is contained in the middlemost
element of the chain in \eqref{chain1}, but not the leftmost.

Let us use 
``$\gamma_1$'' to denote 
the top congruence
$\delta_1+\delta_2\cap\Delta_{\beta,\theta}$
in Figure~\ref{fig3}.
Here, note that since $\delta_1+\delta_2\cap\Delta_{\beta,\theta}$
is strictly above $\delta_1$, it
is indeed of the form $\gamma_1$ for some $\gamma\in\Con(\m a)$
satisfying $\gamma > \delta$. Moreover,
since $\gamma_1=\delta_1+(\delta_2\cap\Delta_{\beta,\theta})
\stackrel{\eqref{delta_gen}}{\leq}
  \delta_1+(\delta_2\cap (\theta_1\times \theta_2))=
  \delta_1+(\theta_1\times \delta_2)=\theta_1$ we
have that
\begin{equation} \label{gamma_location}
\delta < \gamma \leq \theta
\end{equation}
in $\Con(\m a)$.
Note also that 
\begin{equation} \label{gamma_1_location}
\Delta_{\beta,\gamma}\leq \gamma_1\times\gamma_2\leq \gamma_1
\end{equation}
in $\Con(\m a(\beta))$.

We have a chain on the right side of
Figure~\ref{fig3}, namely
$$
\begin{array}{rll}
\delta_1\cap\Delta_{\beta,\theta}&\leq 
(\delta_1\cap \Delta_{\beta,\theta})+(\delta_2\cap\Delta_{\beta,\theta})&
(\delta_2\cap\Delta_{\beta,\theta})\\
&\leq
\Delta_{\beta,\gamma}
+(\delta_1\cap \Delta_{\beta,\theta})+(\delta_2\cap\Delta_{\beta,\theta})\hphantom{AAA}&\Delta_{\beta,\gamma}\\
&\leq
\delta_1+\delta_2\cap\Delta_{\beta,\theta}=\gamma_1& \delta_1.
\end{array}
$$
One sees that, on each of these lines, the congruence
immediately to the right
of the ``$\leq$'' is obtained from the
preceding congruence in the chain by joining
the additional congruence indicated on the same line at the far right.
The claim just made involves three assertions, the first two
of which are formally true. For the third claim,
which is the claim that we obtain $\gamma_1$ when we join
$\delta_1$ to $\Omega=\Delta_{\beta,\gamma}
+(\delta_1\cap \Delta_{\beta,\theta})+(\delta_2\cap\Delta_{\beta,\theta})$,
we use
(\ref{gamma_1_location}) in the last step of the following computation.

$$
\begin{array}{rl}
\delta_1+\Omega&=  
\delta_1+(\Delta_{\beta,\gamma}
+(\delta_1\cap \Delta_{\beta,\theta})+(\delta_2\cap\Delta_{\beta,\theta}))\\
&= \Delta_{\beta,\gamma}
+(\delta_1+(\delta_1\cap \Delta_{\beta,\theta}))+(\delta_2\cap\Delta_{\beta,\theta}))\\
&=
 \Delta_{\beta,\gamma}
+(\delta_1+\delta_2\cap\Delta_{\beta,\theta})\\
&=
 \Delta_{\beta,\gamma}
+\gamma_1\\
&=\gamma_1.\\
\end{array}
$$

Next, we argue that the comparability 
\begin{equation} \label{chain2}
\Psi:=(\delta_1\cap \Delta_{\beta,\theta})+(\delta_2\cap\Delta_{\beta,\theta})
  \leq \Delta_{\beta,\gamma}+(\delta_1\cap \Delta_{\beta,\theta})+(\delta_2\cap\Delta_{\beta,\theta}) =:\Omega,
\end{equation}
is strict.

\begin{clm}\label{strict}
$\Psi<\Omega$.
\end{clm}

\noindent
{\it Proof of Claim~\ref{strict}.}
If $(u,v)\in\gamma-\delta$, then 
the pair $((u,u),(v,v))\in\Delta_{\beta,\gamma}$
will belong to $\Omega$,
since $\Delta_{\beta,\gamma}$ is a summand of $\Omega$.
We shall argue that $((u,u),(v,v))\notin\Psi$.
Our path will be to show that
the $\Psi$-class of $(u,u)$ is contained
in the $\delta_1\times \delta_2$-class of $(u,u)$.
This will suffice, since we have $(u,v)\notin\delta$
and therefore 
$((u,u),(v,v))\notin\delta_1\times \delta_2$,
so we will be able to derive that
$((u,u),(v,v))\notin\Psi$.

Since Figure~\ref{fig2} is a pentagon
in which $\beta\cap (\theta+(\beta\cap\delta))\leq \delta$ holds,
Theorem~\ref{basic_centrality}~(8) guarantees
that the relation $\C C(\theta,\beta;\delta)$ holds.
This property can
be restated in this way:
if $D$ is the set of pairs
$(x,y)\in \m a(\beta)$ that satisfy $(x,y)\in\delta$,
then the subset $D \subseteq \m a(\beta)$
is a union of $\Delta_{\beta,\theta}$-classes.
In other words, 
$(e,f)\in \delta$
and
$(e,f)\stackrel{\Delta_{\beta,\theta}}{\equiv} (g,h)$ 
jointly imply
$(g,h)\in \delta$.
We can apply this information to the intersection congruence
$\delta_1\cap \Delta_{\beta,\theta}$
to derive that if
\begin{itemize}
\item $(e,f)\stackrel{\Delta_{\beta,\theta}}{\equiv} (g,h)$,
\item $(e,f)\in\delta$, and we also have
\item $(e,g)\in\delta$, 
\end{itemize}
then
$(g,h), (f,h)\in\delta$.
(Here $(f,h)\in\delta$ follows from
$f\stackrel{\delta}{\equiv}e\stackrel{\delta}{\equiv}g\stackrel{\delta}{\equiv}h$.)
In conclusion,  if
$(e,f)\in\delta$ and
$((e,f),(g,h)) \in\delta_1\cap \Delta_{\beta,\theta}$,
then necessarily
$(g,h)\in\delta$ and
$((e,f),(g,h)) \in\delta_2\cap \Delta_{\beta,\theta}$.
This establishes that,
when $(e,f)\in\delta$,
the $\delta_1\cap \Delta_{\beta,\theta}$-class
of $(e,f)$ agrees with the
$\delta_2\cap \Delta_{\beta,\theta}$-class of $(e,f)$,
and hence agrees with the $\Psi$-class of $(e,f)$
($\Psi$ is the join of
$\delta_1\cap \Delta_{\beta,\theta}$ and
$\delta_2\cap \Delta_{\beta,\theta}$).
Therefore, when $(e,f)\in\delta$, the $\Psi$-class of $(e,f)$
is contained in the $\delta_1\times\delta_2$-class
of $(e,f)$.
By applying this reasoning to $(e,f)=(u,u)$
we see that
$((u,u),(v,v))\notin\Psi$, since
the $\delta_1\times\delta_2$-class of $(u,u)$
does not contain $(v,v)$.
\cqed
\bigskip

What remains to do to establish
that our congruences form a pentagon is to show that 
(i) $\delta_1+\Psi = \gamma_1$ and
(ii) $\delta_1\cap \Omega=\delta_1\cap\Delta_{\beta,\theta}$.
For the first of these we calculate that
\[
\begin{array}{rl}
  \delta_1+\Psi &= (\delta_1+(\delta_1\cap \Delta_{\beta,\theta}))+(\delta_2\cap\Delta_{\beta,\theta})\\
&= \delta_1+(\delta_2\cap\Delta_{\beta,\theta})\\  
  &=\gamma_1.
  \end{array}
\]
For~(ii), we use the fact $\gamma\leq \theta$ from
\eqref{gamma_location}
to derive that
$\Delta_{\beta,\gamma}\leq \Delta_{\beta,\theta}$.
Therefore all summands in
\[
\Omega = \Delta_{\beta,\gamma}+(\delta_1\cap \Delta_{\beta,\theta})+(\delta_2\cap\Delta_{\beta,\theta})
\]
lie below $\Delta_{\beta,\theta}$. Since one of the summands is
$\delta_1\cap \Delta_{\beta,\theta}$, we obtain 
\[
\delta_1\cap \Delta_{\beta,\theta}\leq \Omega\leq \Delta_{\beta,\theta}.
\]
If we meet this chain throughout with $\delta_1$ we obtain
\[
\delta_1\cap \Delta_{\beta,\theta}\leq \delta_1\cap \Omega\leq \delta_1\cap \Delta_{\beta,\theta},
\]
or $\delta_1\cap \Omega= \delta_1\cap \Delta_{\beta,\theta}$,
which is what (ii) asserts.

We have established the pentagon shape,
so what is left is to establish that the
asserted centralities hold.
In \eqref{gamma_location}
we showed above that $\delta < \gamma\leq \theta$
in $\Con(\m a)$.
Since $\C C(\theta,\theta;\delta)$ holds in
$\Con(\m a)$, we get
$\C C(\theta_1,\theta_1;\delta_1)$ in $\Con(\m a(\beta))$
by
Theorem~\ref{basic_centrality}~(10)
and the Correspondence Theorem.
We then get
$\C C(\gamma_1,\gamma_1;\delta_1)$
by monotonicity
(Theorem~\ref{basic_centrality}~(1)).
This shows that the interval $I[\delta_1,\gamma_1]$
between $\delta_1$ and the top of the pentagon, $\gamma_1$, is abelian.
Using this, we can derive that the interval
$I[\delta_1\cap\Delta_{\beta,\theta},\Omega]$
between
$\Omega = \Delta_{\beta,\gamma}+(\delta_1\cap \Delta_{\beta,\theta})+(\delta_2\cap\Delta_{\beta,\theta})$
and the bottom of the pentagon is also abelian, as follows:
By the facts that $\C C(\gamma_1,\gamma_1;\delta_1)$
and $\Omega\leq \gamma_1$, we can get
$\C C(\Omega,\Omega;\delta_1)$.
We always have 
$\C C(\Omega,\Omega;\Omega)$
according to
Theorem~\ref{basic_centrality}~(8).
Then, by
Theorem~\ref{basic_centrality}~(6),
we get $\C C(\Omega,\Omega;\delta_1\cap\Omega)$,
which is the claim that the interval between
$\Omega$ and the bottom of the pentagon is abelian.
This completes the proof of the theorem
up to relabeling the congruences in Figure~\ref{fig3}.
(In particular, since the bottom element of
Figure~\ref{fig3} is labeled $0$, we should factor
our algebra and take $\m b = \m a(\beta)/(\delta_1\cap \Delta_{\beta,\theta})$.)
\end{proof}

\begin{lm} \label{asymmetry}
Assume that $\m a$ is an algebra whose
commutator operation is not commutative.
Some quotient $\m b$ of $\m a$ will have congruences
$\alpha,\beta\in\Con(\m b)$
such that $[\beta,\alpha]=0<[\alpha,\beta]$.
\end{lm}  

\begin{proof}
For this proof (and later proofs)
we will adopt ``relative commutator'' notation first
introduced above \cite[Theorem~4.22]{order-theoretic}.
This notation is useful for discussing
the relationship between the commutator operation
in $\m a$ and the commutator operations on quotients of $\m a$.
Define 
\[
[\alpha,\beta]_{\varepsilon}:=
\bigcap \{\gamma\;|\; (\gamma\geq \varepsilon)\;\textrm{and}\;\C C(\alpha,\beta;\gamma)\}.
\]
It is easy to see from Theorem~\ref{basic_centrality}~(10)
that this notation has the property that
if $\varepsilon\leq \alpha, \beta$, then
$[\alpha/\varepsilon,\beta/\varepsilon]=
[\alpha,\beta]_{\varepsilon}/\varepsilon$,
so the ordinary (= unsubscripted) commutator operation $[-,-]$
on $\Con(\m a/\varepsilon)$
is reflected by the operation
$[-,-]_{\varepsilon}$ on the interval
$I[\varepsilon,1]$ of $\Con(\m a)$.

If $\m a\in \mathcal V$ has noncommutative
commutator, then it has congruences
$\gamma,\delta\in \Con(\m a)$ such that
$[\gamma,\delta]\not\leq [\delta,\gamma]$.
Set $\varepsilon=[\delta,\gamma]$.
This is a congruence which lies
below both $\gamma$ and $\delta$.
The fact that $[\gamma,\delta]\not\leq [\delta,\gamma]=\varepsilon$
implies that
$\C C(\gamma,\delta;\varepsilon)$ fails, hence
$[\gamma,\delta]_{\varepsilon}\neq \varepsilon=[\delta,\gamma]$.
But $[\gamma,\delta]_{\varepsilon}\geq \varepsilon$
from the definition of the relative commutator notation.
Hence we have 
$[\delta,\gamma]_{\varepsilon} = [\delta,\gamma] < [\gamma,\delta]_{\varepsilon}$.
This means that the algebra $\m a/\varepsilon$
has congruences
$\delta/\varepsilon, \gamma/\varepsilon$ satisfying
$[\delta/\varepsilon,\gamma/\varepsilon]
= 0 < [\gamma/\varepsilon,\delta/\varepsilon]$.
By changing notation to work modulo $\varepsilon$
we have a quotient of $\m a$ with congruences
$\alpha = \gamma/\varepsilon, \beta = \delta/\varepsilon$
satisfying $[\beta,\alpha]=0<[\alpha,\beta]$.
\end{proof}

We will use Lemma~\ref{asymmetry}
in the next result where we connect
left and right
distributivity of the commutator
with commutativity of the commutator.

\begin{thm} \label{distributive_thm}
Let $\mathcal V$ be an arbitrary variety.
\begin{enumerate}
\item If the commutator is left distributive throughout
  $\mathcal V$,
  \[(\forall x, y, z)\;\;[x+y,z]=[x,z]+[y,z],\]
  then it is also commutative throughout $\mathcal V$
  \[(\forall x, y)\;\;[x,y]=[y,x].\]
\item If the commutator is right distributive throughout $\mathcal V$,
  \[(\forall x, y, z)\;\;[x,y+z]=[x,y]+[x,z],\]
  then the commutator
  satisfies the following ``partial commutativity'' on comparable pairs
  of congruences.
 \[(\forall x, y)\;\;(y\leq x) \Rightarrow [x,y]\leq [y,x].\]
\end{enumerate}
\end{thm}

\begin{proof}
We start by proving the contrapositive form of Item~(1),
so assume that the commutator
fails to be commutative throughout $\mathcal V$.
There must be some $\m a\in \mathcal V$ that has congruences
$\alpha,\beta\in \Con(\m a)$ such that
$[\alpha,\beta]\not\leq [\beta,\alpha]$.
By Lemma~\ref{asymmetry} we may assume that
$[\beta,\alpha]=0<[\alpha,\beta]$.

Recall that
$\eta_1=0_1=\{((w,x),(y,z))\in A(\alpha)^2\mid w=y\}$ and 
$\eta_2=0_2 =\{((w,x),(y,z))\in A(\alpha)^2\mid x=z\}$
are the restrictions of the coordinate
projection kernels of $\m a^2$ to the subalgebra
$\m a(\alpha)$.

Let $\delta\colon \m a\to \m a(\alpha)\colon x\mapsto (x,x)$
be the diagonal embedding.
We will write $D$ for the set-theoretic image $\delta(A)$
and $\m d$ for the algebra-theoretic image $\delta(\m a)$
(the
diagonal subuniverse of $\m a(\alpha)$).
For a congruence $\theta\in\Con(\m a)$,
we write $\delta(\theta)$ for
$\{((x,x),(y,y))\in A(\alpha)^2\;|\;(x,y)\in\theta\}$,
which is a congruence on $\m d$.

\begin{clm} \label{left_add1}
The diagonal subuniverse $D\leq \m a(\alpha)$
is a union of singleton 
$([\eta_1,\Delta_{\alpha,\beta}]+[\eta_2,\Delta_{\alpha,\beta}])$-classes.
\end{clm}

\noindent
{\it Proof of Claim~\ref{left_add1}.}
Since $[\beta,\alpha]=0$, the set
$D$ is a union of $\Delta_{\alpha,\beta}$-classes. No
two elements of $D$ are related by $\eta_1$, so 
every element of $D$ is a singleton
$(\eta_1\cap \Delta_{\alpha,\beta})$-class.
Since $[\eta_1,\Delta_{\alpha,\beta}]$ is contained in
$\eta_1\cap \Delta_{\alpha,\beta}$, every element of
$D$ is a singleton $[\eta_1, \Delta_{\alpha,\beta}]$-class.
Similarly every element of
$D$ is a singleton $[\eta_2, \Delta_{\alpha,\beta}]$-class.
Therefore every element of
$D$ is a singleton
$([\eta_1, \Delta_{\alpha,\beta}]+[\eta_2, \Delta_{\alpha,\beta}])$-class.
\cqed
\bigskip

\begin{clm} \label{left_add2}
The diagonal subuniverse $D\leq \m a(\alpha)$
is not a union of singleton 
$([\eta_1+\eta_2,\Delta_{\alpha,\beta}])$-classes.
\end{clm}

\noindent
{\it Proof of Claim~\ref{left_add2}.}
We will show that the restriction of the congruence
$[\eta_1+\eta_2,\Delta_{\alpha,\beta}]$
to $D$ is not the equality relation, and this will prove
that $D$ is not a union of singleton 
$([\eta_1+\eta_2,\Delta_{\alpha,\beta}])$-classes.
For this, notice that
\[
[\eta_1+\eta_2,\Delta_{\alpha,\beta}]|_{\m d}\geq
[(\eta_1+\eta_2)|_{\m d},\Delta_{\alpha,\beta}|_{\m d}]
= [\delta(\alpha),\delta(\beta)]=\delta([\alpha,\beta]) > 0.  
\]
The leftmost inequality is derived from Theorem~\ref{basic_centrality}~(9).
\cqed
\bigskip

Claims \ref{left_add1} and \ref{left_add2}
show that $[\eta_1+\eta_2,\Delta_{\alpha,\beta}]\neq 
[\eta_1,\Delta_{\alpha,\beta}]+[\eta_2,\Delta_{\alpha,\beta}]$,
so the commutator is not
left distributive on $\m a(\alpha)$.
\bigskip

Next we argue the contrapositive form of Item~(2)
of the theorem. Assume that
there is some $\m a\in \mathcal V$ that has congruences
$\alpha,\beta\in \Con(\m a)$ such that (i) $\beta\leq \alpha$
but (ii) $[\alpha,\beta]\not\leq [\beta,\alpha]$.
Consulting the proof of Lemma~\ref{asymmetry},
we see that we may refine these assumptions to
(i) $\beta\leq \alpha$ and
(ii) $[\beta,\alpha]=0<[\alpha,\beta]$.

\begin{clm}
\label{right_add1}
The diagonal subuniverse $D\leq \m a(\alpha)$
is a union of singleton 
$([\eta_1,\eta_2]+[\eta_1,\Delta_{\alpha,\beta}])$-classes.
\end{clm}

\noindent
{\it Proof of Claim~\ref{right_add1}.}
We have $[\eta_1,\eta_2]\leq \eta_1\cap \eta_2 = 0$,
so $[\eta_1,\eta_2]$ is the equality relation on $\m a(\alpha)$
and all $[\eta_1,\eta_2]$-classes are singletons.
Hence the subuniverse $D$ consists of singleton
$[\eta_1,\eta_2]$-classes.
We may copy the proof of Claim~\ref{left_add1} to establish
that $D$ is a union of singleton 
$[\eta_1,\Delta_{\alpha,\beta}]$-classes.
(The situation here is the same as the one there, except
here we have the extra property that $\beta\leq\alpha$.)
It follows that $D$ is a union of singleton 
$([\eta_1,\eta_2]+[\eta_1,\Delta_{\alpha,\beta}])$-classes.
\cqed
\bigskip

\begin{clm}
\label{right_add2}
$D$ is not a union of singleton 
$[\eta_1,\eta_2+\Delta_{\alpha,\beta}]$-classes.
\end{clm}

\noindent
{\it Proof of Claim~\ref{right_add2}.}
For this proof, note that $\eta_2+\Delta_{\alpha,\beta}=\beta_2$.
    
We started the proof by arranging that $[\alpha,\beta]>0$.    
This means that there is an $\alpha,\beta$-matrix
\[
\begin{bmatrix}
t(\wec{a},\wec{u}) &   t(\wec{a},\wec{v}) \\
t(\wec{b},\wec{u}) &   t(\wec{b},\wec{v}) 
\end{bmatrix}=
\begin{bmatrix}
p &  q \\
r &  s
\end{bmatrix}, \;\;\;
\wec{a}\wrel{\alpha}\wec{b},\;\; \wec{u}\wrel{\beta}\wec{v},
\]
with $p=q$ but $r\neq s$.
Consider the $\eta_1, \beta_2$-matrix of $\m a(\alpha)$
\begin{equation}
  \tag{M}\label{MM}
\begin{bmatrix}
t\left((\wec{b},\wec{a}), (\wec{u}, \wec{u})\right)&
t\left((\wec{b},\wec{a}), (\wec{u}, \wec{v})\right)\\
t\left((\wec{b},\wec{b}), (\wec{u}, \wec{u})\right)&
t\left((\wec{b},\wec{b}), (\wec{u}, \wec{v})\right)
\end{bmatrix}=
\begin{bmatrix}
(r, p) &  (r, q) \\
(r, r) &  (r, s).
\end{bmatrix}
\end{equation}
The fact that this truly is an $\eta_1,\beta_2$-matrix
of $\m a(\alpha)$
follows from our assumption that $\beta\leq \alpha$
(and this is the only place in the argument where
this assumption is needed).
Namely, to know that \eqref{MM}
is an $\eta_1,\beta_2$-matrix
of $\m a(\alpha)$
we need to know that $\m a(\alpha)$
contains all pairs of the form
$(b_i,a_i)$, 
$(u_i,u_i)$,
$(b_i,b_i)$, and 
$(u_i,v_i)$.
It is easy to see that $\m a(\alpha)$
contains all pairs of these types except possibly
those of the type
$(u_i,v_i)$.
Such pairs lie in $\beta$, so if $\beta\leq \alpha$
they will also lie in $A(\alpha)=\alpha$.

We have $(r, p)=(r, q)$,
so $\left((r, r),(r, s)\right)$
belongs to $[\eta_1,\beta_2]$.
Since $(r,r) \in D$ and
$(r,s)\notin D$,
the $[\eta_1,\beta_2]$-class of
$(r,r)\in D$ is not a singleton,
hence $D$ is not a union of
singleton $[\eta_1,\beta_2]$-classes.
\cqed
\bigskip

Claims \ref{right_add1} and \ref{right_add2}
show that $[\eta_1,\eta_2+\Delta_{\alpha,\beta}]\neq 
[\eta_1,\eta_2]+[\eta_1,\Delta_{\alpha,\beta}]$,
so the commutator is not
right distributive on $\m a(\alpha)$.
\end{proof}

\section{Main results}\label{main}

Now we prove results which seem to require a ground
Maltsev condition. The most important theorems
of this section will be proved under the assumption
of ``existence of a Taylor term'', \cite[Definition~2.15]{shape}.
A Taylor term for a variety $\mathcal V$
is a term $T(x_1,\ldots,x_{n})$ such that
$\mathcal V$ satisfies the
identity $T(x,\ldots,x)\approx x$
and, for each $i$ between $1$ and $n$, $\mathcal V$ satisfies
some identity of the form $T(\wec{w})\approx T(\wec{z})$
with $w_i \neq z_i$.
Any identity of the form
$T(\wec{w})\approx T(\wec{z})$
with
$w_i \neq z_i$
is called an ``$i$-th place Taylor identity'' of $T$.

Some of the results of this section
will be proved under the stronger
assumptions  ``existence of a difference term'' or
 ``existence of a weak difference term'',
Definition~\ref{left_right}.
The class of varieties with a difference term is 
definable by a Maltsev condition.
The same is true for the class
of varieties with a weak difference term.
The Mal\-tsev conditions
were identified in principle in
\cite{kearnes-szendrei}
in Theorem~4.8 and the paragraph following
the proof of that theorem.
We define
weak and ordinary difference terms next.

\begin{df} \label{left_right}
  Let $\mathcal V$ be a variety.
  A ternary $\mathcal V$-term $t(x,y,z)$ shall be called
\begin{enumerate}
\item  a \emph{right Maltsev term} for $\mathcal V$ if 
${\mathcal V}\models t(x,x,y)\approx y$.
\item  a \emph{left Maltsev term} for $\mathcal V$ if 
${\mathcal V}\models t(x,y,y)\approx x$.
\item  a \emph{Maltsev term} for $\mathcal V$ if it is
  both a right and left Maltsev term.
\item  a \emph{right difference term} for $\mathcal V$ if,
  for any $\m b\in {\mathcal V}$,
  $t^{\m b}(a,a,b) = b$ holds
  whenever the pair $(a,b)$ is contained in
  an abelian
  congruence.
\item  a \emph{left difference term} for $\mathcal V$ if,
  for any $\m b\in {\mathcal V}$,
  $t^{\m b}(a,b,b) = a$ holds
  whenever the pair $(a,b)$ is contained in
  an abelian
  congruence.
\item  a \emph{weak difference term} for $\mathcal V$ if 
  it is both a right and left difference term.
\item  a \emph{difference term} for $\mathcal V$
  if 
it is a right Maltsev term and a left difference term.
\end{enumerate}
\end{df}

This left/right terminology is not standard, but it
is introduced here because the new concept
``right difference term''
will play a role in the proof of
Theorem~\ref{characterization_of_weak}
(e.g., in Claim~\ref{delta_char}).

As we noted in the Introduction,
there is a weakest nontrivial idempotent
Maltsev condition. The class of varieties
defined by this condition is the
the class of varieties with a Taylor term.
As noted before Definition~\ref{left_right},
the classes of varieties with (i) a weak difference term
or (ii) an ordinary difference term 
are also definable by idempotent Maltsev conditions.
Ordinary difference terms are formally stronger
than weak difference terms, which 
are stronger than Taylor terms.
These differences in strength are strict.
The algebra $\m i$ of Example~\ref{simple_ex}
generates a variety that has a Taylor term,
but does not have a weak difference term.
(The justification for this claim is given in that example.)
The semigroup $\mathbb Z_2\times \mathbb S_2$
described in Section~\ref{examples}
generates a variety with a weak difference
term,
but with no difference term.
(The justification for this claim is given in that section.)

Throughout this section, the assumption that
the variety under consideration
has a Taylor term will be used to
invoke the following theorem, which
expresses 
limits on the behavior of the centralizer relation.

\begin{thm} \label{memoir_pentagon}
  \textrm{\rm (\cite[Theorem 4.16(2)(i)]{shape})}
  Let $\mathcal V$ be a variety that has a Taylor term
  and let $\m a$ be a member of $\mathcal V$.
There is no pentagon sublattice of $\Con(\m a)$,
labeled as in Figure~\ref{fig4}, such that $\C C(\beta,\theta;\delta)$ holds.
\begin{figure}[ht]
\begin{center}
\setlength{\unitlength}{1mm}
\begin{picture}(20,33)
\put(0,15){\circle*{1.2}}
\put(10,0){\circle*{1.2}}
\put(10,30){\circle*{1.2}}
\put(20,10){\circle*{1.2}}
\put(20,20){\circle*{1.2}}

\put(10,0){\line(-2,3){10}}
\put(10,0){\line(1,1){10}}
\put(10,30){\line(-2,-3){10}}
\put(10,30){\line(1,-1){10}}
\put(20,20){\line(0,-1){10}}

\put(-4.5,13){$\beta$}
\put(22,9){$\delta$}
\put(22,19){$\theta$}
\put(9,32){$\alpha$}
\end{picture}
\bigskip

\caption{\sc (Theorem~\ref{memoir_pentagon})
  A pentagon in $\textrm{Con}(\m A)$ will
  not satisfy $\C C(\beta,\theta;\delta)$.}\label{fig4}
\end{center}
\end{figure}
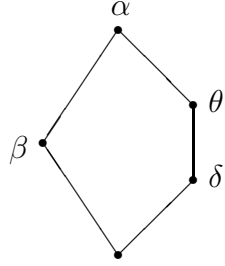
\end{thm}

We restate this theorem in a positive
and formally stronger way.

\begin{thm} \label{memoir_pentagon_2}
  Let $\mathcal V$ be a variety that has a Taylor term
  and let $\m a$ be a member of $\mathcal V$.
Given any pentagon sublattice of $\Con(\m a)$,
labeled as in Figure~\ref{fig5},
$[\alpha,x]_{\delta} = x$ holds for every
$x\in I[\delta,\theta]$.
(Equivalently, if $\delta \leq y < x\leq \theta$, then
$\C C(\alpha,x;y)$ fails.) \hfill $\Box$
\begin{figure}[ht]
\begin{center}
\setlength{\unitlength}{1mm}
\begin{picture}(20,33)
\put(0,15){\circle*{1.2}}
\put(10,0){\circle*{1.2}}
\put(10,30){\circle*{1.2}}
\put(20,10){\circle*{1.2}}
\put(20,15){\circle*{1.2}}
\put(20,20){\circle*{1.2}}

\put(10,0){\line(-2,3){10}}
\put(10,0){\line(1,1){10}}
\put(10,30){\line(-2,-3){10}}
\put(10,30){\line(1,-1){10}}
\put(20,20){\line(0,-1){10}}

\put(-4.5,13){$\beta$}
\put(22,9){$\delta$}
\put(22,14){$x$}
\put(22,19){$\theta$}
\put(9,32){$\alpha$}
\end{picture}
\bigskip

\caption{\sc If $\delta\leq x\leq \theta$, then
$[\alpha,x]_{\delta} = x$.}\label{fig5}
\end{center}
\end{figure}
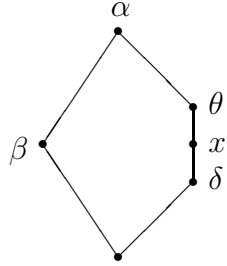
\end{thm}

There are two observations to make to see that
Theorem~\ref{memoir_pentagon}
and Theorem~\ref{memoir_pentagon_2} have the same content.
The first observation is:
(i) $\C C(\beta,\theta;\delta)$ holds in the pentagon of
Figure~\ref{fig4}  if and only if
(ii) $\C C(\alpha,\theta;\delta)$ holds in that pentagon.
One obtains (ii)$\Rightarrow$(i) by monotonicity
of the centralizer in its first place
(Theorem~\ref{basic_centrality}~(1)).
One obtains (i)$\Rightarrow$(ii) by
assuming $\C C(\beta,\theta;\delta)$,
deriving $\C C(\delta,\theta;\delta)$ from 
Theorem~\ref{basic_centrality}~(8), then
deriving $\C C(\beta+\delta,\theta;\delta)$
from 
Theorem~\ref{basic_centrality}~(5).
Since $\alpha = \beta+\delta$, we get that (ii) holds.

The second observation is that if $\delta\leq y <x\leq \theta$, then
$\{\alpha, \beta, x, y, \beta\cap y\}$ is another pentagon
in $\Con(\m a)$. Applying
Theorem~\ref{memoir_pentagon} to this new
pentagon we get that $\C C(\beta,x;y)$ fails.
Using the idea of the preceding paragraph, this
is equivalent to the assertion that 
$\C C(\alpha,x;y)$ fails whenever
$\delta\leq y <x\leq \theta$.
In particular, since $\delta\leq [\alpha,x]_{\delta} \leq x\leq \theta$
and $\C C(\alpha,x;[\alpha,x]_{\delta})$ holds,
we cannot have $[\alpha,x]_{\delta} < x$.
The alternative is that
$[\alpha,x]_{\delta} = x$ whenever $\delta\leq x\leq \theta$.
On the other hand, if we have
$[\alpha,x]_{\delta} = x$ for $x=\theta$, then
we recover the conclusion of
Theorem~\ref{memoir_pentagon}.

The first main theorem of this section
gives a commutator-theoretic characterization
of varieties with a weak difference term
(Theorem~\ref{characterization_of_weak}).
Using this characterization, we shall deduce that
any variety with a Taylor term and
commutative commutator must have a weak
difference term
(Theorem~\ref{commutative_implies_weak}).
The proofs of these two results were developed  
from an analysis
of a simple example, which we describe first.
The inclusion of this example is meant to help guide
the reader through the lengthy proof of
Theorem~\ref{characterization_of_weak}.

\begin{example} \label{simple_ex}
Let $\mathbb R$ be the real line considered as
a $1$-dimensional real vector space.
Let $\mathbb R^{\circ}$ be the reduct of $\mathbb R$
to the idempotent linear operations
of the form $f_r(x,y)=rx+(1-r)y$, $0<r<1$.
Let $\m i$ be the subalgebra of $\mathbb R^{\circ}$
whose universe is the unit interval $I=[0,1]$.
Thus, $\m i = \langle [0,1]; \{f_r(x,y)\;|\;0<r<1\}\rangle$
is a subalgebra of a reduct of an abelian algebra $\mathbb R$,
which makes $\m i$ 
an abelian algebra.
From Definition~\ref{left_right},
the concepts of ``Maltsev term'',
``difference term'', and
``weak difference term'' all coincide for abelian
algebras.
The fact that $I$ is closed under all of the
$f_r$-operations and is not closed
under the unique Maltsev operation
$x-y+z$ of $\mathbb R$ shows that neither $\mathbb R^{\circ}$
nor $\m i$ have Maltsev operations, and
therefore neither $\mathbb R^{\circ}$
nor $\m i$
has a weak difference term.
${\mathcal V}(\m i)$
does have a Taylor term, namely
$T(x,y) = f_{\frac{1}{2}}(x,y)=\frac{1}{2}x+\frac{1}{2}y$.
This is a Taylor operation for ${\mathcal V}(\m i)$,
since $T(x,x)\approx x$ and $T(x,y)\approx T(y,x)$
hold in $\m i$, and the latter is both a first-place
and a second-place Taylor identity for $T(x,y)$ in
${\mathcal V}(\m i)$.

The algebra $\m i$ is free in ${\mathcal V}(\m i)$
over the $2$-element generating set $\{0,1\}$.
We identify some congruences
in $\Con(\m i\times \m i)$ and indicate the sublattice
they generate in Figure~\ref{fig6}.
\begin{figure}[ht]
\begin{center}
\setlength{\unitlength}{1mm}
\begin{picture}(20,35)
\put(0,15){\circle*{1.2}}
\put(15,0){\circle*{1.2}}
\put(15,30){\circle*{1.2}}
\put(30,15){\circle*{1.2}}
\put(22.5,7.5){\circle*{1.2}}

\put(15,0){\line(-1,1){15}}
\put(15,0){\line(1,1){15}}
\put(15,30){\line(-1,-1){15}}
\put(15,30){\line(1,-1){15}}

\put(14,-5){$0$}
\put(14,32){$1$}

\put(-4.5,14){$\eta_1$}
\put(32,14){$\eta_2$}
\put(24,6){$\delta=\cg(((0,0),(1,0)))$}
\end{picture}
\bigskip

\caption{\sc A sublattice of $\Con(\m i\times \m i)$.}\label{fig6}
\end{center}
\end{figure}
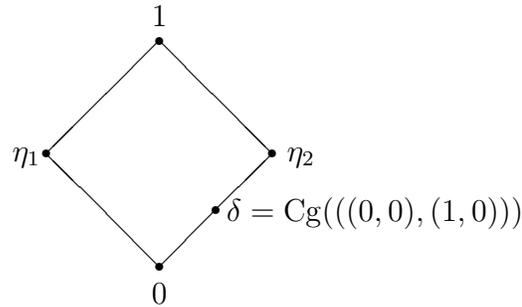

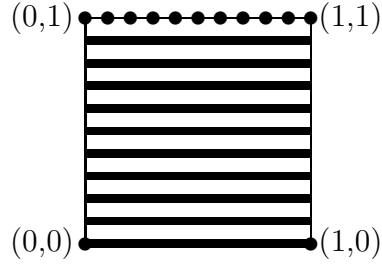
\begin{figure}[ht]
\begin{center}
\setlength{\unitlength}{1mm}
\begin{picture}(20,32)
\thinlines
\put(0,0){\line(1,0){30}}
\put(0,0){\line(0,1){30}}
\put(30,30){\line(-1,0){30}}
\put(30,30){\line(0,-1){30}}

\thicklines
\linethickness{1mm}
\put(0,27){\line(1,0){30}}
\put(0,24){\line(1,0){30}}
\put(0,21){\line(1,0){30}}
\put(0,18){\line(1,0){30}}
\put(0,15){\line(1,0){30}}
\put(0,12){\line(1,0){30}}
\put(0,9){\line(1,0){30}}
\put(0,6){\line(1,0){30}}
\put(0,3){\line(1,0){30}}
\put(0,0){\line(1,0){30}}

\put(0,30){\circle*{1.6}}
\put(3,30){\circle*{1.6}}
\put(6,30){\circle*{1.6}}
\put(9,30){\circle*{1.6}}
\put(12,30){\circle*{1.6}}
\put(15,30){\circle*{1.6}}
\put(18,30){\circle*{1.6}}
\put(21,30){\circle*{1.6}}
\put(24,30){\circle*{1.6}}
\put(27,30){\circle*{1.6}}
\put(30,30){\circle*{1.6}}

\put(0,0){\circle*{1.6}}
\put(30,00){\circle*{1.6}}

\put(-10,-1){(0,0)}
\put(-10,29){(0,1)}
\put(31,-1){(1,0)}
\put(31,29){(1,1)}
\end{picture}
\bigskip

\caption{\sc The partition of $\m i\times \m i$ induced by $\delta=\cg(((0,0),(1,0)))$.}\label{fig7}
\end{center}
\end{figure}

The projection kernels are the congruences
$\eta_1=\cg(((0,0),(0,1)),((1,0),(1,1)))$
and
$\eta_2=\cg(((0,0),(1,0)),((0,1),(1,1)))$.
The congruence $\eta_1$ partitions the ``square''
$\m i\times \m i$
into congruence classes that are ``vertical lines''
and $\eta_2$ partitions 
$\m i\times \m i$
into congruence classes that are ``horizontal lines''.
The interesting congruence is $\delta=\cg(((0,0),(1,0)))$.
The partition of
$\m i\times \m i$ it yields is depicted in Figure~\ref{fig7}.
In the partition depicted in Figure~\ref{fig7},
all congruence classes of $\delta$ agree with those
of $\eta_2$ except the class that is the ``top line'',
$X:=\m i\times \{1\}$.
The top line $X$ is a single 
$\eta_2$-class, and hence a union
of $\delta$-classes, but 
$\delta$ restricts to be the equality relation
on the top line $X$ 
while $\eta_2$ restricts to be the total relation.

This unusual structure for $\delta$
can be exploited the following way.
The operation
$T(x,y) = f_{\frac{1}{2}}(x,y)=\frac{1}{2}x+\frac{1}{2}y$
may be used to create a $1,\eta_2$-matrix

\medskip

$$
\begin{bmatrix}
T\left((1,0), (0,1)\right)&
T\left((1,0), (1,1)\right)\\
T\left((1,1), (0,1)\right)&
T\left((1,1), (1,1)\right)\\
\end{bmatrix}=
\begin{bmatrix}
(.5,.5) &  (1,.5)\\
(.5, 1) &  (1,1) 
\end{bmatrix}
$$

\medskip

\noindent
where the elements on the top row are $\delta$-related
while the elements on the bottom row are not.
This matrix witnesses that $\neg \C C(1,\eta_2;\delta)$.
In the quotient $(\m i\times \m i)/\delta$
we must have $[\overline{1},\overline{\eta}_2]>0$
where $\overline{\eta}_2:=\eta_2/\delta$
and $\overline{1}:=1/\delta$.
It is possible to argue that 
$[\overline{\eta}_2,\overline{1}]=0$,
and therefore that
$[\overline{\eta}_2,\overline{1}]\neq [\overline{1},\overline{\eta}_2]$.
Although this example is special,
the location of noncommutativity in $\mathcal V$
is general as we shall see in the proofs
of the next two results.
\end{example}

\begin{thm} \label{characterization_of_weak}
The following are equivalent for a variety $\mathcal V$.
\begin{enumerate}
\item $\mathcal V$ has a weak difference term.
\item Whenever $\m a\in \mathcal V$ and $\alpha\in\Con(\m a)$
  is abelian, the interval $I[0,\alpha]$ consists of permuting
  equivalence relations.
\item Whenever $\m a\in \mathcal V$ and $\alpha\in\Con(\m a)$
  is abelian, the interval $I[0,\alpha]$ is modular.
\item Whenever $\m a\in \mathcal V$ and $\alpha\in\Con(\m a)$,
  there is no pentagon labeled as in
  Figure~\ref{fig8} with $[\alpha,\alpha]=0$.
(No ``spanning pentagon'' in $I[0,\alpha]$ if $\alpha$ is abelian.)  
\item Whenever $\m a\in \mathcal V$ and $\alpha\in\Con(\m a)$,
there is no pentagon labeled as in
  Figure~\ref{fig8} where $[\alpha,\alpha]=0$ 
  and $\C C(\theta,\alpha;\delta)$.
\begin{figure}[ht]
\begin{center}
\setlength{\unitlength}{1mm}
\begin{picture}(20,33)
\put(0,15){\circle*{1.2}}
\put(10,0){\circle*{1.2}}
\put(10,30){\circle*{1.2}}
\put(20,10){\circle*{1.2}}
\put(20,20){\circle*{1.2}}

\put(10,0){\line(-2,3){10}}
\put(10,0){\line(1,1){10}}
\put(10,30){\line(-2,-3){10}}
\put(10,30){\line(1,-1){10}}
\put(20,20){\line(0,-1){10}}

\put(-4.5,13){$\beta$}
\put(22,9){$\delta$}
\put(22,19){$\theta$}
\put(9,32){$\alpha$}
\put(9,-5){$0$}
\end{picture}
\bigskip

\caption{\sc Forbidden sublattice if both $[\alpha,\alpha]=0$ and
  $\C C(\theta,\alpha;\delta)$ hold.}\label{fig8}
\end{center}
\end{figure}
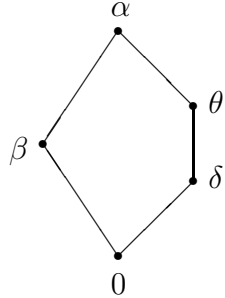
\end{enumerate}  
\end{thm}

\begin{proof}

[$(1)\Rightarrow(2)$]
Assume that
$t(x,y,z)$ is a weak difference term for $\mathcal V$,
$\m a\in \mathcal V$, and $\alpha\in\Con(\m a)$
is abelian.
If $\sigma,\tau\in I[0,\alpha]$ and
$a\stackrel{\sigma}{\equiv} b \stackrel{\tau}{\equiv} c$,
then
$a=t^{\m a}(a,b,b)\stackrel{\tau}{\equiv}
t^{\m a}(a,b,c) \stackrel{\sigma}{\equiv} t^{\m a}(b,b,c)=c$.
This is all that is needed to verify that
$\sigma\circ\tau\subseteq \tau\circ\sigma\;\;(\subseteq \sigma\circ\tau)$.

[$(2)\Rightarrow(3)$]
Every lattice of permuting equivalence relations is modular.

[$(3)\Rightarrow(4)$]
A lattice is modular if and only if it has no pentagon sublattice.
Notice that Item~(3) asserts that when $\alpha$
is abelian, then $I[0,\alpha]$ contains no
pentagon sublattice. Item~(4) asserts something
slightly more: when $\alpha$
is abelian, then $I[0,\alpha]$ contains no
``spanning pentagon'' sublattice, by which we mean
a pentagon whose bottom is $0$
and whose top is $\alpha$.

[$(4)\Rightarrow(5)$]
Item~(5) is identical to Item~(4) except the
pentagons in Item~(5) are more restricted.
In Item~(5) we assert no abelian interval contains
a spanning pentagon
\emph{which also satisfies $\C C(\theta,\alpha;\delta)$}.

[$(5)\Rightarrow(1)$]
This implication is the only nontrivial claim
of the theorem.
We prove it in the contrapositive form:
if $\mathcal V$ does not have a weak
difference term, then it will
contain an algebra with a pentagon
like the one described in Item~(5).

 Let
  \begin{itemize}
\item $\m f=\m f_{\mathcal V}(x,y)$ be the free $\mathcal V$-algebra
  over the set $\{x,y\}$.
\item   $\theta =\Cg {\m f}(x,y)$.
\item $\overline{\m f} = \m f/[\theta,\theta]$.
\item $\overline{x} = x/[\theta,\theta], \overline{y} = y/[\theta,\theta]$.
\item $\overline{\theta} = \theta/[\theta,\theta]=\Cg {\overline{\m f}}(\overline{x},\overline{y})$.
  \end{itemize}

It follows from properties
of the commutator (Theorem~\ref{basic_centrality}~(10))
that $[\overline{\theta},\overline{\theta}]=0$.
  
\begin{clm} \label{free_abelian}
  {\rm ($\overline{\theta}$ is the ``free principal abelian
    congruence'' in $\mathcal V$)}
Suppose that $\m b\in {\mathcal V}$ and $\beta\in\Con(\m b)$
satisfies $[\beta,\beta]=0$. For any $(a,b)\in \beta$
there is a unique homomorphism
$\overline{\varphi}\colon \overline{\m f}\to \m b$
satisfying $\overline{x}\mapsto a, \overline{y}\mapsto b$.  
\end{clm}

\emph{Proof of Claim~\ref{free_abelian}.}
Given  $(a, b)\in\beta$,
let $\varphi\colon \m f=\m f_{\mathcal V}(x,y)\to \m b$
be the homomorphism determined
by $x\mapsto a, y\mapsto b$. Since $\C C(\beta,\beta;0)$
holds in $\m b$,
$\C C(\varphi^{-1}(\beta),\varphi^{-1}(\beta);\ker(\varphi))$
holds in $\overline{\m f}$. Since $(x,y)\in \varphi^{-1}(\beta)$,
we have $\C C(\theta,\theta;\ker(\varphi))$ for
$\theta = \cg(x,y)$ by monotonicity.
Therefore $[\theta,\theta]\leq \ker(\varphi)$ holds.
We may factor $\varphi$ modulo 
$[\theta,\theta]$ as
\[\varphi\colon \m f\to \m f/[\theta,\theta]=
  \overline{\m f}\stackrel{\overline{\varphi}}{\to} \m b\]
  where the first map is the natural map.
  This yields the desired map
  $\overline{\varphi}\colon \overline{\m f}\to \m b\colon
  \overline{x}\mapsto a, \overline{y}\mapsto b$.
  The uniqueness is a consequence of the fact that
  $\overline{\m f}$ is generated by
  $\{\overline{x},\overline{y}\}$,
  so any homomorphism with domain $\overline{\m f}$
  is uniquely determined
  by its values on this set.
\cqed
\bigskip

Let $\m a$ be the subalgebra of
$\overline{\m f}\times \overline{\m f}$
that is generated by the set
\[
G=\{(\overline{x},\overline{x}), (\overline{y},\overline{x}), (\overline{x},\overline{y}), (\overline{y},\overline{y})\}.
\]
Since $\overline{\m f}$ is generated by $\{\overline{x}, \overline{y}\}$,
the universe of $\m a$ is the
reflexive, symmetric, compatible,
binary relation (or ``tolerance'')
generated by the pair $(\overline{x},\overline{y})$
on the algebra $\overline{\m f}$.
Since $G\subseteq \overline{\theta}$ we have
$A\subseteq \overline{\theta}$, so in fact
$\m a\leq \overline{\m f}(\overline{\theta})\leq \overline{\m f}\times
\overline{\m f}$.
This is enough to draw some conclusions about $\m a$.
The most tautological conclusion that follows from
$A\subseteq \overline{\theta}$
is that if $(\overline{u},\overline{v})\in A$,
then $(\overline{u},\overline{v})\in \overline{\theta}$,
so any pair in $\m a$ generates an abelian congruence
in $\overline{\m f}$.
This fact will be used without reminders.
A less obvious conclusion is that,
since $\overline{\theta}$ is an abelian congruence
that relates $\overline{x}$ to $\overline{y}$
in $\overline{\m f}$, the congruence
$\overline{\theta}_1\times \overline{\theta}_2$
is an abelian congruence on
$\overline{\m f}\times\overline{\m f}$ whose
restriction to $\m a$ is an abelian congruence
that
relates any two elements of $G$.

Let $\eta_1$ and $\eta_2$ be the coordinate projection
kernels of $\overline{\m f}\times \overline{\m f}$
restricted to $\m a$.
Following Example~\ref{simple_ex}, let $\delta =
\Cg {\m a}(((\overline{x},\overline{x}),(\overline{y},\overline{x})))$.
Let $X = (\overline{y},\overline{y})/\eta_2$ be the 
$\eta_2$-class of $(\overline{y},\overline{y})$.
The set $X$ plays the role of the ``top edge''
in Example~\ref{simple_ex}.
Elements of $X$ have the form $(\overline{P},\overline{y})$
for certain elements $\overline{P}\in \overline{\m f}$
satisfying 
$\overline{P}\stackrel{\overline{\theta}}{\equiv}
\overline{y}\stackrel{\overline{\theta}}{\equiv}\overline{x}$
in $\overline{\m f}$.

\begin{clm} \label{delta_char}
{\rm  (Characterization of $\delta|_X$)}
  Two pairs $(\overline{P},\overline{y})$ and $(\overline{Q},\overline{y})$
  lying in $X$
  are $\delta$-related if and only if there is a ternary
  $\mathcal V$-term $t_{\overline{P},\overline{Q}}(x,y,z)$ such that
\begin{enumerate}
\item[($\dagger$)]
  $t_{\overline{P},\overline{Q}}$ is a right difference
  term for $\mathcal V$ (cf.\ Definition~\ref{left_right}), and
\item[($\ddagger$)]  $t_{\overline{P},\overline{Q}}^{\overline{\m f}}(\overline{y},\overline{x},\overline{P})=\overline{Q}$.
\end{enumerate}
\end{clm}

\emph{Proof of Claim~\ref{delta_char}.}
Recall that $\delta$ is the principal
congruence on $\m a$ that is generated by
the pair
$((\overline{x},\overline{x}),(\overline{y},\overline{x}))$.
Since $X$ is an $\eta_2$-class and $\delta\leq \eta_2$,
if $(\overline{P},\overline{y})$ and $(\overline{Q},\overline{y})$
  lying in $X$
  are $\delta$-related, then they are connected
  by a Maltsev chain that lies entirely inside $X$.
  A link of such a Maltsev chain has the form
  $(f((\overline{x},\overline{x})),f((\overline{y},\overline{x})))$
or  $(f((\overline{y},\overline{x})),f((\overline{x},\overline{x})))$  
  for some polynomial $f\in \textrm{Pol}_1(\m a)$.
  Since $\m a$ is generated by
  the set $G=\{(\overline{x},\overline{x}), (\overline{y},\overline{x}), (\overline{x},\overline{y}), (\overline{y},\overline{y})\}$,
  we may assume that the parameters of the polynomial $f$
  lie in $G$, and hence we may write
  \[
  f((z,w))=s^{\m a}((z,w),(\overline{x},\overline{x}), (\overline{y},\overline{x}), (\overline{x},\overline{y}), (\overline{y},\overline{y}))
  \]
  for some $5$-ary $\mathcal V$-term $s$. If
$(f((\overline{x},\overline{x})),f((\overline{y},\overline{x})))=  
((\overline{R},\overline{y}),(\overline{S},\overline{y}))$
is a link in a Maltsev chain in $X$, then there
must exist such a $5$-ary term $s$ such that
\[
\begin{array}{rl}
  s^{\m a}((\overline{x},\overline{x}),(\overline{x},\overline{x}), (\overline{y},\overline{x}), (\overline{x},\overline{y}), (\overline{y},\overline{y})) &= (\overline{R},\overline{y})\\
s^{\m a}((\overline{y},\overline{x}),(\overline{x},\overline{x}), (\overline{y},\overline{x}), (\overline{x},\overline{y}), (\overline{y},\overline{y})) &= 
(\overline{S},\overline{y}),
\end{array}
\]
which simplifies to the coordinate equations
\begin{equation}\label{coordinatewise}
\begin{array}{rl}
s^{\overline{\m f}}(\overline{x},\overline{x},\overline{y},\overline{x},\overline{y}) &= \overline{R}\\
s^{\overline{\m f}}(\overline{y},\overline{x},\overline{y},\overline{x},\overline{y}) &= \overline{S}\\
s^{\overline{\m f}}(\overline{x},\overline{x},\overline{x},\overline{y},\overline{y}) &= \overline{y}.
\end{array}
\end{equation}
Let $t_{\overline{R},\overline{S}}(x,y,z) = s(x,y,y,z,z)$. 

\begin{subclm} \label{deltacharsubclaim1}
  {\rm (The $(\dagger)$-part of ``only if''
    when $((\overline{R},\overline{y}),(\overline{S},\overline{y}))$
    is a Maltsev link)}  \;\;\;

\begin{center}    
  $t_{\overline{R},\overline{S}}$ is a right difference
  term for $\mathcal V$. 
\end{center}
  \end{subclm}

\emph{Proof of Subclaim~\ref{deltacharsubclaim1}.}
We derive from the definition of $t_{\overline{R},\overline{S}}$
and the third equation in
\eqref{coordinatewise} that
$t^{\overline{\m f}}_{\overline{R},\overline{S}}(\overline{x},\overline{x},\overline{y})=
s^{\overline{\m f}}(\overline{x},\overline{x},\overline{x},\overline{y},\overline{y}) =\overline{y}$.
The claim then follows from the fact that
$\overline{\theta} = \Cg {\overline{\m f}}(\overline{x},\overline{y})$
is the ``free principal abelian congruence'' in $\mathcal V$.
Specifically, given any $\m b\in {\mathcal V}$ and
any pair $(a,b)$ from an abelian congruence
of $\m b$, Claim~\ref{free_abelian} guarantees
a unique homomorphism
$\overline{\varphi}\colon \overline{\m f}\to \m b$
satisfying $\overline{x}\mapsto a, \overline{y}\mapsto b$.
Applying this homomorphism to the equality
$t^{\overline{\m f}}_{\overline{R},\overline{S}}(\overline{x},\overline{x},\overline{y})=\overline{y}$ yields 
$t^{\m b}_{\overline{R},\overline{S}}(a,a,b)=b$.
This is all that is required to prove that 
$t_{\overline{R},\overline{S}}$ is right difference term for $\mathcal V$.
(Cf.\ Definition~\ref{left_right}~(4).)
  \cqed
  \bigskip

\begin{subclm} \label{deltacharsubclaim2}
  {\rm (The $(\ddagger)$-part of ``only if''
    when $((\overline{R},\overline{y}),(\overline{S},\overline{y}))$
    is a Maltsev link)}  \;\;\;

\begin{center}  
  $t_{\overline{R},\overline{S}}^{\overline{\m f}}(\overline{y},\overline{x},\overline{R})=\overline{S}$.
\end{center}  
\end{subclm}

\emph{Proof of Subclaim~\ref{deltacharsubclaim2}.}
The following matrix is a $\overline{\theta},\overline{\theta}$-matrix
in $\overline{\m f}$:
\begin{equation}\label{tricky_matrix}
\begin{bmatrix}
  s^{\overline{\m f}}(\overline{x},\overline{x},\overline{y},\overline{x},\overline{y})&
s^{\overline{\m f}}(\overline{x},\overline{x},\overline{x},\overline{R},\overline{R})\\
  s^{\overline{\m f}}(\overline{y},\overline{x},\overline{y},\overline{x},\overline{y})&
s^{\overline{\m f}}(\overline{y},\overline{x},\overline{x},\overline{R},\overline{R})
\end{bmatrix}=
\begin{bmatrix}
\overline{R} &  t_{\overline{R},\overline{S}}^{\overline{\m f}}(\overline{x},\overline{x},\overline{R})\\
\overline{S} &  t_{\overline{R},\overline{S}}^{\overline{\m f}}(\overline{y},\overline{x},\overline{R})
\end{bmatrix}
=
\begin{bmatrix}
\overline{R} &  \overline{R} \\
\overline{S} &  t_{\overline{R},\overline{S}}^{\overline{\m f}}(\overline{y},\overline{x},\overline{R})
\end{bmatrix}.
\end{equation}
The justifications for the claims
that these matrices are equal and that the
leftmost is an 
$\overline{\theta},\overline{\theta}$-matrix
follow from the facts that 
(i) $(\overline{x},\overline{y}), (\overline{y},\overline{y}),
(\overline{R},\overline{y}), (\overline{S},\overline{y})$
belong to $X\subseteq A$, hence
$\overline{x},\overline{y},\overline{R},\overline{S}$
belong to the same $\overline{\theta}$-class,
(ii) equations \eqref{coordinatewise} hold,
(iii) 
$t_{\overline{R},\overline{S}}(x,y,z):=s(x,y,y,z,z)$, and
(iv) 
$t_{\overline{R},\overline{S}}$ is a right difference
term for $\mathcal V$
(Subclaim~\ref{deltacharsubclaim1}).
Item (i) is enough to show that the leftmost
matrix is a $\overline{\theta},\overline{\theta}$-matrix,
Items (ii) and (iii) are enough to show that the leftmost
matrix reduces to the middlemost,
and Item (iv) is enough to show that the middlemost
matrix reduces to the rightmost one.

Since $\overline{\theta}$ is abelian and the top
row of the matrix in \eqref{tricky_matrix} is constant,
the bottom row must also be constant.
\cqed
\bigskip

\begin{subclm} \label{only_if_delta_char}
The ``only if'' part of Claim~\ref{delta_char} holds.
\end{subclm}

\emph{Proof of Subclaim~\ref{only_if_delta_char}.}
Subclaims~\ref{deltacharsubclaim1} and \ref{deltacharsubclaim2}
prove the ``only if'' part of
Claim~\ref{delta_char}
    when $((\overline{R},\overline{y}),(\overline{S},\overline{y}))$
    is a Maltsev link.
Maltsev links are Maltsev chains of length $1$.
Here we prove that the ``only if'' part holds
for Maltsev chains of any length. For this,
let $\Omega$ be the relation on $X$ consisting of all pairs
$((\overline{P},\overline{y}), (\overline{Q},\overline{y}))$
which satisfy ($\dagger$) and
($\ddagger$) of Claim~\ref{delta_char}.
$\Omega$ will contain all pairs
$((\overline{R},\overline{y}), (\overline{S},\overline{y}))$
that are Maltsev links, so to prove this subclaim
it suffices to prove
that $\Omega$ is an equivalence relation on $X$.
\bigskip

($\Omega$ is {\bf reflexive})
Given
$((\overline{P},\overline{y}), (\overline{P},\overline{y}))$,
the third projection term $t_{\overline{P},\overline{P}}(x,y,z) := z$
witnesses membership in $\Omega$.
To see this, note that both ($\dagger$) and
($\ddagger$) are trivial when
$\overline{P}=\overline{Q}$ and $t_{\overline{P},\overline{P}}(x,y,z)=z$:

\begin{enumerate}
\item[($\dagger$)]
  $t_{\overline{P},\overline{P}}(x,y,z):=z$
  is a right difference term for $\mathcal V$.
  (It is even right Maltsev, which is formally stronger.)
\item[($\ddagger$)]
$t_{\overline{P},\overline{P}}^{\overline{\m f}}(\overline{y},\overline{x},\overline{P})=\overline{P}$.
\end{enumerate}
\bigskip

($\Omega$ is {\bf symmetric})
Assume that the ternary term
$s_{\overline{P},\overline{Q}}(x,y,z)$ witnesses membership in $\Omega$ for the pair
$((\overline{P},\overline{y}), (\overline{Q},\overline{y}))$.
We argue that the ternary term $t_{\overline{P},\overline{Q}}(x,y,z):=s_{\overline{P},\overline{Q}}(y,x,z)$ obtained by swapping the first two variables in $s_{\overline{P},\overline{Q}}(x,y,z)$
witnesses membership in $\Omega$ for
$((\overline{Q},\overline{y}), (\overline{P},\overline{y}))$:
\begin{enumerate}
\item[($\dagger$)]
$t_{\overline{P},\overline{Q}}$ is a right difference term for $\mathcal V$.
\end{enumerate}
\medskip

Reason:
Choose $(a,b)$ generating an abelian
congruence in some $\m b\in {\mathcal V}$.
$t_{\overline{P},\overline{Q}}^{\m b}(a,a,b)=s_{\overline{P},\overline{Q}}^{\m b}(a,a,b)=b$.
\medskip

\begin{enumerate}
\item[($\ddagger$)]
  $t_{\overline{P},\overline{Q}}^{\overline{\m f}}(\overline{y},\overline{x},\overline{Q})=\overline{P}$.
\end{enumerate}
  
\medskip

Reason:
We know from ($\dagger$) for $s_{\overline{P},\overline{Q}}$
and from the fact that 
$s_{\overline{P},\overline{Q}}(x,y,z)$ witnesses membership in $\Omega$ for 
$((\overline{P},\overline{y}), (\overline{Q},\overline{y}))$
that
\[s_{\overline{P},\overline{Q}}^{\overline{\m f}}(\overline{y},\overline{x},\overline{P})
=\overline{Q}=
s_{\overline{P},\overline{Q}}^{\overline{\m f}}(\overline{x},\overline{x},\overline{Q}).
\]
The following is a $\overline{\theta},\overline{\theta}$-matrix in
$\overline{\m f}$.
\begin{equation}\label{tricky_matrix2}
\begin{bmatrix}
s_{\overline{P},\overline{Q}}^{\overline{\m f}}(\overline{y},\overline{x},\overline{P})&
s_{\overline{P},\overline{Q}}^{\overline{\m f}}(\overline{x},\overline{x},\overline{Q})\\
s_{\overline{P},\overline{Q}}^{\overline{\m f}}(\overline{y},\overline{y},\overline{P})&
s_{\overline{P},\overline{Q}}^{\overline{\m f}}(\overline{x},\overline{y},\overline{Q})
\end{bmatrix}=
\begin{bmatrix}
\overline{Q}&
\overline{Q}\\
\overline{P}&
s_{\overline{P},\overline{Q}}^{\overline{\m f}}(\overline{x},\overline{y},\overline{Q})
\end{bmatrix}=
\begin{bmatrix}
\overline{Q}&
\overline{Q}\\
\overline{P}&
t_{\overline{P},\overline{Q}}^{\overline{\m f}}(\overline{y},\overline{x},\overline{Q})
\end{bmatrix}.
\end{equation}
Since $\overline{\theta}$ is abelian and this matrix
is constant on the first row we must have
$t_{\overline{P},\overline{Q}}^{\overline{\m f}}(\overline{y},\overline{x},\overline{Q})=\overline{P}$, which is the statement to be proved.
\bigskip

($\Omega$ is {\bf transitive})
Assume that the ternary term
$r_{\overline{P},\overline{Q}}(x,y,z)$ witnesses membership in $\Omega$ for the pair
$((\overline{P},\overline{y}), (\overline{Q},\overline{y}))$
and that $s_{\overline{Q},\overline{W}}(x,y,z)$ witnesses membership in $\Omega$ for the pair
$((\overline{Q},\overline{y}), (\overline{W},\overline{y}))$.
We shall argue that the ternary term $t_{\overline{P},\overline{W}}(x,y,z):=s_{\overline{Q},\overline{W}}(x,y,r_{\overline{P},\overline{Q}}(x,y,z))$
  witnesses that
  $((\overline{P},\overline{y}), (\overline{W},\overline{y}))$
    belongs to $\Omega$.
\begin{enumerate}
\item[($\dagger$)]
  $t_{\overline{P},\overline{W}}$
  is a right difference term for $\mathcal V$. 
\end{enumerate}
\medskip

Reason:
Choose $(a,b)$ generating an abelian
congruence in some $\m b\in {\mathcal V}$.
We have
$t_{\overline{P},\overline{W}}^{\m b}(a,a,b)=
s_{\overline{Q},\overline{W}}^{\m b}(a,a,
r_{\overline{P},\overline{Q}}^{\m b}(a,a,b))=
s_{\overline{Q},\overline{W}}^{\m b}(a,a,b)=b$.
\medskip

\begin{enumerate}
\item[($\ddagger$)]  $t^{\overline{\m f}}(\overline{y},\overline{x},\overline{P})=\overline{W}$.
\end{enumerate}
\medskip

Reason:
$t_{\overline{P},\overline{W}}^{\overline{\m f}}(\overline{y},\overline{x},\overline{P})=
s_{\overline{Q},\overline{W}}^{\overline{\m f}}(\overline{y},\overline{x},
r_{\overline{P},\overline{Q}}^{\overline{\m f}}(\overline{y},\overline{x},\overline{P}))=
s_{\overline{Q},\overline{W}}^{\overline{\m f}}(\overline{y},\overline{x},\overline{Q})=
\overline{W}$.
\cqed
\bigskip

\begin{subclm}\label{if}
{\rm (``if'' statement in Claim~\ref{delta_char})}  
Assume that
$(\overline{P},\overline{y})$ and $(\overline{Q},\overline{y})$
lie in $X$ and there is a ternary
  $\mathcal V$-term $t(x,y,z)$ such that
\begin{enumerate}
\item[($\dagger$)]
  $t_{\overline{P},\overline{Q}}$ is a right difference
  term for $\mathcal V$, and
\item[($\ddagger$)]  $t_{\overline{P},\overline{Q}}^{\overline{\m f}}(\overline{y},\overline{x},\overline{P})=\overline{Q}$.
\end{enumerate}
  Then $(\overline{P},\overline{y})$ and $(\overline{Q},\overline{y})$
  are $\delta$-related.
\end{subclm}

\emph{Proof of Subclaim~\ref{if}.}
Since $t_{\overline{P},\overline{Q}}$ is a right difference term
and $\overline{x}, \overline{y}, \overline{P}, \overline{Q}$
belong to a single class of the abelian
congruence $\overline{\theta}$ we have both
$t_{\overline{P},\overline{Q}}^{\overline{\m f}}(\overline{x},\overline{x},\overline{P})=\overline{P}$
and
$t_{\overline{P},\overline{Q}}^{\overline{\m f}}(\overline{x},\overline{x},\overline{y})=\overline{y}$.
Hence, working with pairs in $\m a$,
\[
\begin{array}{rl}
  (\overline{P},\overline{y})&=t_{\overline{P},\overline{Q}}^{\m a}(\underline{(\overline{x},\overline{x})},(\overline{x},\overline{x}),(\overline{P},\overline{y}))\\
  \vphantom{|}&\\
  &\stackrel{\delta}{\equiv}
  t_{\overline{P},\overline{Q}}^{\m a}(\underline{(\overline{y},\overline{x})},(\overline{x},\overline{x}),(\overline{P},\overline{y}))\\
  \vphantom{|}&\\  
&= (\overline{Q},\overline{y}).
\end{array}
\]
In moving from the first line to the second we have underlined
the only change, indicating where we replaced
$(\overline{x},\overline{x})$ with the $\delta$-related
pair $(\overline{y},\overline{x})$. In moving from
the second line to the third we made coordinatewise use
of ($\dagger$) and ($\ddagger$) for $t_{\overline{P},\overline{Q}}$.
\cqed
\bigskip

This completes the proof of Claim~\ref{delta_char}.
\cqed
\bigskip

\begin{clm} \label{delta_neq_eta_2}
$((\overline{x},\overline{y}),(\overline{y},\overline{y}))\notin\delta.$
\end{clm}

\emph{Proof of Claim~\ref{delta_neq_eta_2}.}
Assume instead
that $((\overline{x},\overline{y}),(\overline{y},\overline{y}))\in\delta.$
Then, for $\overline{P}=\overline{x}$ and $\overline{Q}=\overline{y}$,
we have
$((\overline{P},\overline{y}),(\overline{Q},\overline{y}))\in\delta|_X$.
Claim~\ref{delta_char} guarantees the existence of
a ternary term $t_{\overline{x},\overline{y}}(x,y,z)$ such that
($\dagger$) $t_{\overline{x},\overline{y}}$ is a right difference term for $\mathcal V$
and ($\ddagger$)
\begin{equation}\label{1-sided-diff}
t_{\overline{x},\overline{y}}(\overline{y},\overline{x},\overline{x})=\overline{y}.
\end{equation}
In Claim~\ref{free_abelian} 
we showed that 
$\overline{x}$ and $\overline{y}$
are generators of the ``free principal abelian congruence''
in $\mathcal V$.
We used this information in Subclaim~\ref{deltacharsubclaim1}
to prove that the equality
$t_{\overline{R},\overline{S}}(\overline{x},\overline{x},\overline{y})=\overline{y}$
in $\overline{\m f}$
suffices to prove that
$t_{\overline{R},\overline{S}}$ is a right difference term for $\mathcal V$.
The same argument can be applied here to prove that
\eqref{1-sided-diff}
suffices to prove that
$t_{\overline{x},\overline{y}}$ is a left difference term for $\mathcal V$.
But now $t_{\overline{x},\overline{y}}$ is both a right
and a left difference term, which
contradicts our initial assumption that 
$\mathcal V$ has no weak difference term (= left and right difference term).
\cqed
\bigskip

At present we have no understanding of how $\delta$
behaves off of the set $X$. To deal with this,
we enlarge $\delta$ to
the largest congruence below $\eta_2$
that ``behaves like $\delta$ on $X$''.
Since $X$ is an $\eta_2$-class, all
congruences $\gamma$ satisfying
$\gamma\leq \eta_2$ have the property
that $X$ is a union of $\gamma$-classes.
This implies that the set of those $\gamma$
satisfying
\begin{enumerate}
  \item[(i)] $\gamma\leq \eta_2$
    and
  \item[(ii)] $\gamma|_X\subseteq \delta|_X$
\end{enumerate}
contains $\delta$ and is closed under complete join.
Let $\varepsilon\in \Con(\m a)$ be the
join of all congruences
satisfying (i) and (ii).
Since $\delta$ is a joinand,
we get that $\delta\leq \varepsilon\leq\eta_2$ and
$\varepsilon|_X = \delta|_X$.
Since $((\overline{x},\overline{y}),(\overline{y},\overline{y}))\in\eta_2$
and, by Claim~\ref{delta_neq_eta_2}, 
$((\overline{x},\overline{y}),(\overline{y},\overline{y}))\notin\delta|_X$, we get that $\eta_2|_X\neq \delta|_X=\varepsilon|_X$, and therefore
$\delta\leq \varepsilon < \eta_2$.
This is enough to imply that $\{\eta_1, \eta_2, \varepsilon\}$ generates
a pentagon in $\Con(\m a)$ that is labeled like the one in
Figure~\ref{fig9}.
In this figure, it might happen that $\delta=\varepsilon$,
but no other pair of differently-labeled
congruences in the figure could be equal.

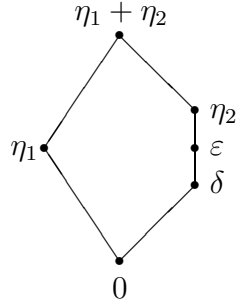
\begin{figure}[ht]
\begin{center}
\setlength{\unitlength}{1mm}
\begin{picture}(20,33)
\put(0,15){\circle*{1.2}}
\put(10,0){\circle*{1.2}}
\put(10,30){\circle*{1.2}}
\put(20,10){\circle*{1.2}}
\put(20,15){\circle*{1.2}}
\put(20,20){\circle*{1.2}}

\put(10,0){\line(-2,3){10}}
\put(10,0){\line(1,1){10}}
\put(10,30){\line(-2,-3){10}}
\put(10,30){\line(1,-1){10}}
\put(20,20){\line(0,-1){10}}

\put(-4.5,14){$\eta_1$}
\put(22,9){$\delta$}
\put(22,14){$\varepsilon$}
\put(22,19){$\eta_2$}
\put(4,32){$\eta_1+\eta_2$}
\put(9,-5){$0$}
\end{picture}
\bigskip

\caption{\sc $\delta\leq \varepsilon < \eta_2,\;\; \delta|_X=\varepsilon|_X$.}\label{fig9}
\end{center}
\end{figure}

\begin{clm} \label{varepsilon_char}
{\rm  (Characterization of $\varepsilon$)}
\[
\varepsilon = \{(a,b)\in\eta_2\;|\; \forall f\in\textrm{\rm Pol}_1(\m a)\;\big(f(a), f(b))\in X\Rightarrow (f(a),f(b))\in\delta\big)\}.
\]
\end{clm}

\emph{Proof of Claim~\ref{varepsilon_char}.}
It is not difficult to see that the relation on the right hand
side of the equality symbol is (i) an equivalence relation contained
in $\eta_2$ (ii) that is closed under the application of unary
polynomials and (iii) whose restriction to $X$ is contained in $\delta$.
It is also clear that the relation on the right hand
side of the equality symbol contains
all other relations with these three properties,
hence it is the largest congruence $\gamma\leq \eta_2$
satisfying $\gamma|_X\subseteq \delta|_X$.
This is enough to conclude that
the relation on the right hand
side of the equality symbol is $\varepsilon$.
\cqed
\bigskip

\begin{clm} \label{cent1}
{\rm   (i)} $\eta_1+\eta_2$ is abelian and {\rm (ii)}
  $\C C(\eta_2,\eta_1+\eta_2;\varepsilon)$ holds.
\end{clm}

\emph{Proof of Claim~\ref{cent1}.}
For (i), we already noted in the paragraph
preceding Claim~\ref{delta_char},
$\overline{\theta}_1\times \overline{\theta}_2$ is an abelian
congruence of $\m a$. Since $\m a\leq \overline{\m f}(\overline{\theta})$
we have
$\eta_1, \eta_2\leq \overline{\theta}_1\times \overline{\theta}_2$.
From this we derive that
$\eta_1+\eta_2\leq \overline{\theta}_1\times \overline{\theta}_2$
and then that $\eta_1+\eta_2$ is an abelian congruence
of $\m a$.

For (ii), 
we must show that, given an $\eta_2,(\eta_1+\eta_2)$-matrix
\begin{equation}\label{varepsilon_matrix}
\begin{bmatrix}
t(\wec{a},\wec{u}) &   t(\wec{a},\wec{v}) \\
t(\wec{b},\wec{u}) &   t(\wec{b},\wec{v}) 
\end{bmatrix}=
\begin{bmatrix}
p &  q \\
r &  s
\end{bmatrix}, \;\;\;
\end{equation}
if $p\stackrel{\varepsilon}{\equiv} q$, then
$r\stackrel{\varepsilon}{\equiv} s$.
We shall assume that $p\stackrel{\varepsilon}{\equiv} q$ and
$r\not\stackrel{\varepsilon}{\equiv} s$ and argue
to a contradiction.

We do have $r\stackrel{\eta_2}{\equiv} p
\stackrel{\varepsilon}{\equiv} q\stackrel{\eta_2}{\equiv} s$.
Since $\varepsilon\leq\eta_2$, all elements $p, q, r, s$
are $\eta_2$-related. If indeed 
$r\not\stackrel{\varepsilon}{\equiv} s$,
then by Claim~\ref{varepsilon_char}
there is a unary polynomial $f(x)$
of $\m a$ such that $f(r), f(s)\in X$ and 
$f(r)\not\stackrel{\delta}{\equiv} f(s)$.
We may apply $f$ to the matrix
in \eqref{varepsilon_matrix} to obtain another
$\eta_2,(\eta_1+\eta_2)$-matrix
\begin{equation}\label{varepsilon_matrix_2}
\begin{bmatrix}
f(p) &  f(q) \\
f(r) &  f(s)
\end{bmatrix}. \;\;\;
\end{equation}
The new matrix has the same properties
as the old one, except now we also have
all entries in $X = (\overline{y},\overline{y})/\eta_2$. Let us write
$(\overline{P},\overline{y}), (\overline{Q},\overline{y}), (\overline{R},\overline{y}), (\overline{S},\overline{y})$
for $f(p), f(q), f(r), f(s)$.

By Claim~\ref{delta_char}, the fact that
$(\overline{P},\overline{y})
\stackrel{\delta|_X}{\equiv}(\overline{Q},\overline{y})$ holds
yields a right difference term
$t_{\overline{P},\overline{Q}}$
such that
$
t_{\overline{P},\overline{Q}}(\overline{y},\overline{x},\overline{P})=\overline{Q}.
$
We must have 
$
t_{\overline{P},\overline{Q}}(\overline{y},\overline{x},\overline{R})\neq\overline{S},
$
or the same right difference term
would yield
$(\overline{R},\overline{y})\stackrel{\delta|_X}{\equiv}(\overline{S},\overline{y})$,
which is false.

Now consider the $\eta_2,(\eta_1+\eta_2)$-matrix
\begin{equation}\label{varepsilon_matrix_3}
t_{\overline{P},\overline{Q}}\left(
\begin{bmatrix}
(\overline{y},\overline{y}) &  (\overline{x},\overline{y}) \\
(\overline{y},\overline{y}) &  (\overline{x},\overline{y}) \\
\end{bmatrix},
\begin{bmatrix}
(\overline{x},\overline{y}) &  (\overline{x},\overline{y}) \\
(\overline{x},\overline{y}) &  (\overline{x},\overline{y}) \\
\end{bmatrix},
\begin{bmatrix}
(\overline{P},\overline{y}) &  (\overline{Q},\overline{y}) \\
(\overline{R},\overline{y}) &  (\overline{S},\overline{y})
\end{bmatrix}
\right)
=
\begin{bmatrix}
(\overline{Q},\overline{y}) & (\overline{Q},\overline{y}) \\
(t_{\overline{P},\overline{Q}}(\overline{y},\overline{x},\overline{R}),\overline{y})
  & (\overline{S},\overline{y})
\end{bmatrix}.
\end{equation}
The rightmost matrix witnesses that $\C C(\eta_2,\eta_1+\eta_2;0)$ fails,
since the top row is constant while
the bottom row is not, since
$
t_{\overline{P},\overline{Q}}(\overline{y},\overline{x},\overline{R})\neq\overline{S}.
$
But the failure of 
$\C C(\eta_2,\eta_1+\eta_2;0)$ contradicts
$\C C(\eta_1+\eta_2,\eta_1+\eta_2;0)$, which we established in part (i)
of this claim. This completes the proof of (ii).
\cqed
\bigskip

We have constructed the desired pentagon, so to
complete the proof of $\neg (1)\;\Rightarrow\; \neg(5)$
of Theorem~\ref{characterization_of_weak} we just have to
explain how to relabel the elements of the pentagon.
Relabel each congruence in the sequence
$(\eta_1+\eta_2,\eta_1,\eta_2,\varepsilon,0)$
of the pentagon of in Figure~\ref{fig9}
with the corresponding label in 
$(\alpha,\beta,\theta,\delta,0)$
to obtain the pentagon in Figure~\ref{fig8}.
From Claim~\ref{cent1} we have that
$\alpha$ is abelian and $\C C(\theta,\alpha;\delta)$,
as desired.
\end{proof}

The next result represents a half-step toward proving
that a variety with a Taylor term and commutative commutator
has a difference term.
It turns out that we only need to assume that
commutativity of the commutator on comparable
pairs of congruences to prove this result.

\begin{thm} \label{commutative_implies_weak}
  If $\mathcal V$ has a Taylor term and,
  whenever $x \leq y$ in $\Con(\m a)$ for
  $\m a\in {\mathcal V}$,
  the equation $[x,y]=[y,x]$
  is satisfied, then
  $\mathcal V$ must have a weak difference term.
\end{thm}

\begin{proof}
We shall assume that $\mathcal V$ has a Taylor term and
does not have a weak difference term and argue
that $\mathcal V$ contains an algebra
with a noncommutative commutator.
Our construction produces an algebra in $\mathcal V$
which has a pair of comparable
congruences $x\leq y$ such that
$[x,y]\neq [y,x]$.

Since we have assumed that $\mathcal V$ does not have
a weak difference term, Theorem~\ref{characterization_of_weak}
guarantees that there is some
algebra $\m a\in {\mathcal V}$ whose congruence
lattice contains a pentagon
\begin{figure}[ht]
\begin{center}
\setlength{\unitlength}{1mm}
\begin{picture}(20,33)
\put(0,15){\circle*{1.2}}
\put(10,0){\circle*{1.2}}
\put(10,30){\circle*{1.2}}
\put(20,10){\circle*{1.2}}
\put(20,20){\circle*{1.2}}

\put(10,0){\line(-2,3){10}}
\put(10,0){\line(1,1){10}}
\put(10,30){\line(-2,-3){10}}
\put(10,30){\line(1,-1){10}}
\put(20,20){\line(0,-1){10}}

\put(-4.5,13){$\beta$}
\put(22,9){$\delta$}
\put(22,19){$\theta$}
\put(9,32){$\alpha$}
\put(9,-5){$0$}
\end{picture}
\bigskip

\caption{\sc Both $[\alpha,\alpha]=0$ and
  $\C C(\theta,\alpha;\delta)$ hold.}\label{fig10}
\end{center}
\end{figure}
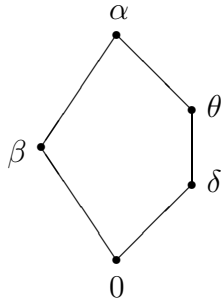
\noindent
where $\alpha$ is abelian and
$\C C(\theta,\alpha;\delta)$ holds.
We claim that the algebra $\m a/\delta\in{\mathcal V}$
has noncommutative commutator.
Specifically, we claim that
\begin{equation}\label{noncom}
[\,\theta/\delta, \alpha/\delta\,]\;=\;0\;\neq \;
[\,\alpha/\delta,\theta/\delta\,].
\end{equation}

To see this, first observe that since
$\C C(\theta,\alpha;\delta)$ holds we have
$[\,\theta/\delta,\alpha/\delta\,]=0$ from
Theorem~\ref{basic_centrality}~(10),
and this is the equality in \eqref{noncom}.

It remains to prove the inequality
$[\,\alpha/\delta,\theta/\delta\,]\neq 0$ from \eqref{noncom}.
If instead we had equality, then again by
Theorem~\ref{basic_centrality}~(10) we would have that
$\C C(\alpha,\theta;\delta)$ holds.
By monotonicity (Theorem~\ref{basic_centrality}~(1)),
we would have
$\C C(\beta,\theta;\delta)$ holds.
This contradicts Theorem~\ref{memoir_pentagon}.
Thus, for $x=\theta/\delta$ and $y=\alpha/\delta$
we have $x\leq y$ and $[x,y]\neq [y,x]$.
\end{proof}

Next we begin a sequence of results to strengthen the conclusion
of Theorem~\ref{commutative_implies_weak}
from ``weak difference term'' to ``difference term''.
The argument is completed in Theorem~\ref{main1} below.

\begin{thm}\label{stable}
Assume that $\mathcal V$ has a weak difference term.
If $\m a$ in $\mathcal V$ has congruences satisfying
\begin{enumerate}
\item $\alpha\geq \theta\geq \delta$, and
\item $[\alpha,\theta]=0$, then
\end{enumerate}
$\C C(\theta,\alpha;\delta)$ holds.
\end{thm}

\begin{proof}
Let $d(x,y,z)$ be some fixed weak difference term for $\mathcal V$.

This is a proof by contradiction, so
assume that the hypotheses hold
and that the conclusion
$\C C(\theta,\alpha;\delta)$ fails.
This failure is witnessed by 
a $\theta,\alpha$-matrix
\[
\begin{bmatrix}p&q\\r&s
\end{bmatrix}=
\begin{bmatrix}
t(\wec{a},\wec{c})&t(\wec{a},\wec{d})\\
t(\wec{b},\wec{c})&t(\wec{b},\wec{d})
\end{bmatrix}
\]
where $(a_i,b_i)\in\theta$, $(c_j,d_j)\in\alpha$,
$(p,q)\in\delta$,
and $(r,s)\notin\delta$.
Since $(r,p)\in\theta$,
$(q,s)\in\theta$, and $(p,q)\in \delta\leq \theta$ we have
$r\stackrel{\theta}{\equiv} p 
\stackrel{\delta}{\equiv} q
\stackrel{\theta}{\equiv} s$, so
$p, q, r$ and $s$ are all $\theta$-related.
From the hypotheses 
(1) $\alpha\geq \theta$, and
(2) $[\alpha,\theta]=0$, we derive that
$[\theta,\theta]=0$ by monotonicity.
This implies that $d(x,y,z)$
acts like a Maltsev operation
on the $\theta$-class containing $p,q,r,s$.
Let $t'(\wec{x},\wec{y})
= d(t(\wec{x},\wec{y}),t(\wec{x},\wec{c}),t(\wec{b},\wec{c}))$.
We have an $\alpha,\theta$-matrix
\[
\begin{bmatrix}
  t'(\wec{a},\wec{c})&t'(\wec{b},\wec{c})\\
  t'(\wec{a},\wec{d})&t'(\wec{b},\wec{d})
\end{bmatrix}
=
\begin{bmatrix}
  d(p,p,r)&d(r,r,r)\\
  d(q,p,r)&d(s,r,r)
\end{bmatrix}
=
\begin{bmatrix}
  r&r\\
  d(q,p,r)&s
\end{bmatrix}.
\]
In moving from the middlemost matrix
to the rightmost matrix we use the fact that
$d$ acts like a Maltsev operation on the
$\theta$-class containing $p,q,r,s$.
Since this matrix is an $\alpha,\theta$-matrix,
the top row is constant, and
$[\alpha,\theta]=0$, we derive that
the
bottom row is constant,
i.e. $d(q,p,r)=s$. This proves
that $s = d(q,\underline{p},r)\stackrel{\delta}{\equiv} d(q,\underline{q},r) = r$,
which contradicts our earlier assumptions that
$(p,q)\in\delta$ and $(r,s)\notin\delta$.
\end{proof}

\begin{thm}
\label{diff_char}
  \textrm{\rm (\cite[Theorem 3.3(a)$\Leftrightarrow$(b)]{diff})}
If $\mathcal V$ is a variety, then the following conditions are equivalent.
\begin{enumerate}
\item[(a)] $\mathcal V$ has a difference term.
\item[(b)]
  For each $\m a\in {\mathcal V}$ the solvability relation
    is a congruence on $\Con(\m a)$ which is preserved by homomorphisms.
 Furthermore, whenever the pentagon $\m N_5$ is a sublattice of
 $\Con(\m a)$, then the critical interval is neutral.
\end{enumerate}
\end{thm}

 (A congruence interval $I[\delta,\theta]$ is called \emph{neutral}
 if it contains no nontrivial abelian subinterval $I[x,y]$,
 equivalently if $[y,y]_x=y$ whenever $\delta\leq x<y\leq \theta$.)
    
\begin{cor}
\label{diff_char2}
A variety has a difference term if and only if it has a weak
difference term and whenever $\m N_5$ is a sublattice of
 $\Con(\m a)$, then the critical interval is neutral.
\end{cor}

\begin{proof}
  This result will be derived from Theorem~\ref{diff_char}.
  Assume first that $\mathcal V$ has a difference term.
  This term is also a weak difference
  term for $\mathcal V$. Also,
whenever $\m N_5$ is a sublattice of
$\Con(\m a)$, then the critical interval must be neutral
by Theorem~\ref{diff_char} (a)$\Rightarrow$(b).

Conversely, assume that $\mathcal V$ has a weak difference
term and whenever $\m N_5$ is a sublattice of
$\Con(\m a)$, then the critical interval is neutral.
Our goal is to prove that $\mathcal V$
has a difference term.
According to Theorem~\ref{diff_char}, what remains
to show is that for each $\m a\in {\mathcal V}$ the solvability relation
is a congruence on $\Con(\m a)$ which is preserved by homomorphisms.
This property always holds for
varieties with a weak difference term,
as we now explain.

Following the notation of Chapter 6 of \cite{shape},
write $\alpha\lhd\beta$ to mean that $\alpha\leq\beta$
and that $\C C(\beta,\beta;\alpha)$ holds
($\beta$ is abelian over $\alpha$).
The solvability relation is defined
so that $\gamma{}\solv{} \delta$ holds exactly when
there is a finite chain
\[
\gamma\cap\delta =
\varepsilon_0\lhd \cdots
\lhd
\varepsilon_n = \gamma+\delta.
\]
A related notion, $\infty$-solvability, is defined in \cite[Definition~6.5]{shape}
the same way, but with finite chains replaced by
continuous well-ordered chains.
That is, $\alpha\lhdlhd\beta$ means
$\alpha\leq \beta$ and there is continuous well-ordered chain
$(\varepsilon_{\mu})_{\mu<\kappa+1}$
such that $\alpha=\varepsilon_0$,
$\varepsilon_{\mu}\lhd \varepsilon_{\mu+1}$
for $\mu<\kappa+1$ and $\varepsilon_{\lambda} =
\bigcup_{\mu<\lambda} \varepsilon_{\mu}$ for limit $\lambda\leq\kappa$,
and $\bigcup_{\mu<\kappa} \varepsilon_{\mu}=\beta$.
Then $\gamma$ is $\infty$-solvably related to $\delta$
if $\gamma\cap\delta \lhdlhd \gamma+\delta$.
The claims that we need to prove here for solvability
were proved for $\infty$-solvability in \cite{shape},
and the proofs given there work for our purposes.
In particular, Lemma~6.10 of \cite{shape} proves
that, for any $\gamma$,
\begin{itemize}
\item if $\alpha\lhd \beta$, then
$\alpha\cap\gamma\lhd \beta\cap \gamma$ and
$\alpha+\gamma\lhd \beta+\gamma$.
\end{itemize}
This is the technical lemma which is used
in Theorem~6.11 of \cite{shape} to prove that the
$\infty$-solvability relation is compatible with finite meets
and arbitrary joins.
The same arguments 
show that the ordinary solvability relation is compatible with finite meets
and finite joins.

The fact that the $\infty$-solvability relation is preserved by homomorphisms
is proved in Theorem~6.19~(1) of \cite{shape}.
The same proof works here for the ordinary solvability relation.
This completes the proof (sketch) for the converse.
\end{proof}

The next lemma refines
the statement of Theorem~\ref{memoir_pentagon_2}
for varieties that do \emph{not} have a difference term.
This lemma will be used in the proofs of
Theorem~\ref{commutative_plus_weak},
Theorem~\ref{main1.5},
and
Theorem~\ref{main4}.

\begin{lm} \label{commutative_plus_Taylor_lm}
  If $\mathcal V$ has a Taylor term and
  does not have a difference term, then $\mathcal V$
  contains an algebra $\m a$ with congruences
  labeled as in Figure~\ref{fig11} satisfying the following
  commutator conditions:
\begin{enumerate}
\item[(1)] $[\alpha,\theta]=0$,
\item[(2)] $\C C(\theta,\alpha;\delta)$, and
\item[(3)] $[\alpha,x]_{\delta}=x$ for all $x\in I[\delta,\theta]$.
\end{enumerate}  
\begin{figure}[ht]
\begin{center}
\setlength{\unitlength}{1mm}
\begin{picture}(20,33)
\put(0,15){\circle*{1.2}}
\put(10,0){\circle*{1.2}}
\put(10,30){\circle*{1.2}}
\put(20,10){\circle*{1.2}}
\put(20,15){\circle*{1.2}}
\put(20,20){\circle*{1.2}}

\put(10,0){\line(-2,3){10}}
\put(10,0){\line(1,1){10}}
\put(10,30){\line(-2,-3){10}}
\put(10,30){\line(1,-1){10}}
\put(20,20){\line(0,-1){10}}

\put(-4.5,13){$\beta$}
\put(22,9){$\delta$}
\put(22,14){$x$}
\put(22,19){$\theta$}
\put(9,32){$\alpha$}
\put(9,-5){$0$}
\end{picture}
\bigskip

\caption{\sc $\mathcal V$ has a Taylor term but not a difference term.}\label{fig11}
\end{center}
\end{figure}
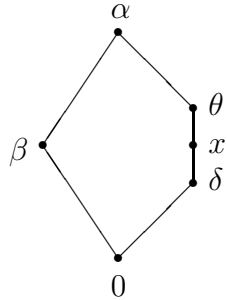
\end{lm}

\begin{proof}
  We split the proof into two cases depending on whether
  $\mathcal V$ has a weak difference term.

  For the first case, assume that $\mathcal V$
  \emph{does not} have a weak difference term.
  Theorem~\ref{characterization_of_weak}~(5) guarantees
  that some $\m a\in \mathcal V$ has a pentagon
  in its congruence lattice, labeled as in Figure~\ref{fig11},
  where (i) $[\alpha,\alpha]=0$ and (ii) $\C C(\theta,\alpha;\delta)$.
  Since (i) is stronger than Item (1) of this lemma statement
  (by monotonicity of the commutator), and (ii) is the same as (2),
  we only have to verify Item (3) of the lemma statement.
That follows from   
Theorem~\ref{memoir_pentagon_2}, since $\mathcal V$ has a Taylor term.

For the second case, assume that $\mathcal V$ \emph{does}
have a weak difference term.
We still assume that $\mathcal V$
\emph{does not}
have a difference term.
By Corollary~\ref{diff_char2},
the fact that $\mathcal V$ does not
have  a difference term
means that some $\m a\in {\mathcal V}$ has a pentagon
in its congruence lattice, labeled as in Figure~\ref{fig11},
where the critical interval $I[\delta,\theta]$ 
is not neutral.
  The nonneutrality means that $[x,x]_{\delta}<x$
  for some $x\in I[\delta,\theta]$. 
  We have
  \[
\delta \leq [x,x]_{\delta}<x \leq \theta,
\]
so we can alter the pentagon by shrinking
$I[\delta,\theta]$ 
to
$I[[x,x]_{\delta},x]$.
This produces a new pentagon with abelian critical
interval. Change to this pentagon and change notation.
That is, assume that $\{\beta,\delta,\theta\}$
generates a pentagon in $\Con(\m a)$,
labeled as in Figure~\ref{fig4}, 
with critical 
interval $I[\delta,\theta]$ where $\C C(\theta,\theta;\delta)$ holds.

By Theorem~\ref{memoir_pentagon}, $\C C(\beta,\theta;\delta)$ fails,
since $\mathcal V$ has a Taylor term.
Now, citing Theorem~\ref{better_pentagons} (Better pentagons),
we may further adjust our pentagon so that
$\C C(\theta,\theta;0)$ holds.
We have $\C C(\beta,\theta;0)$
by
Theorem~\ref{basic_centrality}~(8),
so for $\alpha = \beta+\theta$
we have $\C C(\alpha,\theta;0)$ by
Theorem~\ref{basic_centrality}~(5).
This may be written as $[\alpha,\theta]=0$,
which is Item (1) of the lemma statement.
Since $\alpha\geq \theta\geq \delta$, 
we may invoke Theorem~\ref{stable} to derive
that $\C C(\theta,\alpha;\delta)$ holds.
This is Item~(2) of the lemma statement.
We derive Item (3) from Theorem~\ref{memoir_pentagon_2},
as we did in the first case of this proof.
\end{proof}

\begin{thm} \label{commutative_plus_weak}
If $\mathcal V$ has a weak difference term and 
the commutator is commutative on pairs
of comparable congruences, then
$\mathcal V$ has a difference term.
\end{thm}

\begin{proof}
Assume that $\mathcal V$ has a weak difference term and 
does not have a difference term.
The hypotheses of Lemma~\ref{commutative_plus_Taylor_lm}
hold, so some $\m a\in {\mathcal V}$ has
a pentagon in its congruence lattice with
congruences labeled as in Figure~\ref{fig11},
and which satisfies the commutator conditions (1), (2), and (3)
of Lemma~\ref{commutative_plus_Taylor_lm}.
In $\m a/\delta$ the congruences
$x=\alpha/\delta$ and $y=\theta/\delta$
satisfy $0<y <x$,
$[y,x]=0$ (by Item (2) of that lemma)
and $[x,y]=y$ (by Item (3) of that lemma).
\end{proof}

The next theorem is one of the primary results
of this article.

\begin{thm} \label{main1}
  The following are equivalent for a variety $\mathcal V$.
\begin{enumerate}  
\item[(1)]
  $\mathcal V$ has a difference term
\item[(2)] $\mathcal V$ has a Taylor term and has
  commutative commutator.
\item[(3)]  $\mathcal V$ has a Taylor term and
  the commutator operation is commutative
  on   pairs of comparable congruences.
\end{enumerate}  
\end{thm}

\begin{proof}
The class of varieties that have a difference term
is definable by a nontrivial idempotent Maltsev condition.
(The reason that the class
of varieties with a difference term
is definable by an idempotent Maltsev condition is explained in
the paragraph following
the proof of Theorem~4.8 of \cite{kearnes-szendrei}.
A specific Maltsev
condition defining the class of varieties with a difference term
is in \cite[Section~4]{kissterm}.)
The weakest nontrivial idempotent Maltsev condition
  is the one that asserts the existence of a Taylor term.
  Thus, the implication (1)$\Rightarrow$(2)
  follows from Lemma~2.2 of \cite{diff}, which proves
  that a variety with a difference term has commutative commutator.
  Item~(2) is formally stronger than Item~(3),
  so it remains to prove that
  Item~(3) implies Item~(1).
For this, 
combine Theorems~\ref{commutative_implies_weak}  
and \ref{commutative_plus_weak}.
\end{proof}
\bigskip

Now we turn to an examination of distributivity
of the commutator.
You
will recall that we proved
some results about the left or right distributivity
of the commutator in Theorem~\ref{distributive_thm}.
The results obtained there were left/right-asymmetric, but that
asymmetry disappears when a Taylor term
is present, as we establish with the next two results.

\begin{thm} \label{refinement}
If $\mathcal V$ has a Taylor term, then for any 
$\m A\in \mathcal V$ and any congruences $\alpha, \beta\in \Con(\m a)$
the following are equivalent:
\begin{enumerate}
\item[(a)]
  $[\beta,\alpha]=0$ and $[\alpha,\alpha\cap\beta]=0$.
\item[(b)] $\Delta_{\alpha,\beta}$ is disjoint from the
  coordinate projection kernels of $\m a(\alpha)$.
\item[(c)]  $[\alpha,\beta] = 0$ and $[\beta,\alpha\cap\beta]=0$.
\item[(d)]
$\Delta_{\beta,\alpha}$ is disjoint from the
  coordinate projection kernels of $\m a(\beta)$.
\end{enumerate}

In particular, if $\mathcal V$ has a Taylor term, then
$\mathcal V$ satisfies the commutator congruence
quasi-identity
\[
[\beta,\alpha]=0 \;\;\wrel{\&}\;\; [\alpha,\alpha\cap\beta]=0
\;\wrel{\Longrightarrow}\;
[\alpha,\beta]=0.
\]
\end{thm}

\begin{proof}
This proof is a refinement of the proof of
Lemma~4.4 of \cite{kearnes-szendrei}.

To prove that (a) implies (b), it suffices 
to prove that (a) implies that $\eta_2\cap\Delta_{\alpha,\beta}=0$.
For then, by interchanging the two coordinates of $\m a(\alpha)$,
the same argument will show that $\eta_1\cap\Delta_{\beta,\beta}=0$ also.
Let $\theta = \eta_2\cap\Delta_{\alpha,\beta}\in\Con(\m a(\alpha))$.
Assuming (a), that $[\beta,\alpha]=0$ holds,
the diagonal of $\m a(\alpha)$
is a union of $\Delta_{\alpha,\beta}$-classes. No
two distinct diagonal elements are related by $\eta_2$,
and $\theta = \eta_2\cap \Delta_{\alpha,\beta}$, 
so it follows that each element of the diagonal
of $\m a(\alpha)$ is a singleton $\theta$-class.
Choose an arbitrary pair 
$((a,c),(b,c))\in\theta$.
Let $T(x_1,\ldots,x_{n})$ be a Taylor term for $\mathcal V$.
Consider a first-place Taylor identity
$T(x,\wec{w})\approx T(y,\wec{z})$ where
$\wec{w}, \wec{z}\in \{x,y\}^{n-1}$.
Substitute $b$ for all occurrences of $x$ and $c$
for all occurrences of $y$.
This yields $T(b,\bar{u}) = T(c,\bar{v})$
where all $u_i$ and $v_i$ are in $\{b,c\}$.
Since $(b,c)\in\m A(\alpha)$ we have
$b\stackrel{\alpha}{\equiv} c$, hence
$(u_i,v_i)\in\alpha$ for all $i$. This implies
that $p((x,y)) = (T(x,\bar{u}),T(y,\bar{v}))$
is a unary polynomial of $\m a(\alpha)$.
The equation $T(x,\wec{w})\approx T(y,\wec{z})$
implies that 
$p((b,c))$ lies on the diagonal of $\m a(\alpha)$.
The element $p((a,c))$ is $\theta$-related to
$p((b,c))$, and each element of the diagonal
is a singleton $\theta$-class, therefore
$p((a,c)) = p((b,c))$. This has the consequence
that $T(a,\bar{u}) = T(b,\bar{u})$
where each $u_i\in \{b,c\}$.
Now, since
$((a,c),(b,c))\in\theta
\leq\Delta_{\alpha,\beta}\leq \beta_1\times\beta_2$
we get that $(a,b)\in\beta$. 
Since $(a,c)$ and $(b,c)$ are elements of our algebra
we have $a\stackrel{\alpha}{\equiv} c\stackrel{\alpha}{\equiv}b$,
so $(a,b)\in\alpha$.
Together, the last two sentences show that
$(a,b)\in\alpha\cap\beta$.
Now, applying $[\alpha,\alpha\cap\beta]=0$
to the equality $T(a,\bar{u}) = T(b,\bar{u})$,
we deduce that 
$T(a,\bar{y}) = T(b,\bar{y})$
for any $\bar{y}$ whose entries
lie in the $\alpha$-class containing $a, b$ and $c$.

The argument we just gave concerning $a$, $b$ and $T$,
which showed that $T(a,\bar{y}) = T(b,\bar{y})$
whenever each $y_i$ is in the $\alpha$-class
containing $a, b$, and $c$ is an argument which
works in each of the $n$ variables of $T$.
That is,
\[T(y_1,\ldots,y_{i-1},a,y_{i+1},\ldots,y_n) = 
T(y_1,\ldots,y_{i-1},b,y_{i+1},\ldots,y_n)\]
for each $i$ and any choice of values for 
$y_1,\ldots,y_n$ in the $\alpha$-class of $a, b$, and $c$.
Therefore, using the
fact that $T$ is idempotent, we have
\[
a=T(a,a,\ldots,a)=T(b,a,\ldots,a)=\cdots=T(b,b,\ldots,b)=b.
\]
This proves that $(a,c)=(b,c)$.
Since $((a,c),(b,c))\in\theta$
was arbitrarily chosen, $\theta =\eta_2\cap\Delta_{\alpha,\beta}= 0$,
completing the proof that (a) implies (b).

Now we argue by contraposition that (b) implies (c).
Assume that (c) fails because $[\alpha,\beta]>0$.
There is an $\alpha,\beta$-matrix
\[
\begin{bmatrix}
p&q\\r&s\\
\end{bmatrix}
=
\begin{bmatrix}
t(\wec{e},\wec{u})&  t(\wec{e},\wec{v})\\
t(\wec{f},\wec{u})&  t(\wec{f},\wec{v})\\
\end{bmatrix}
\]
where $(e_i,f_i)\in \alpha$, $(u_j,v_j)\in\beta$,
$p=q$, and $r\neq s$.
The pair
\[
((r,p), (s,q))
=(t((\wec{f},\wec{e}),(\wec{u},\wec{u})), 
t((\wec{f},\wec{e}),(\wec{v},\wec{v})))
\]
belongs to $\eta_2$ (since $p=q$)
and $\Delta_{\alpha,\beta}$
(since $((u_i,u_i), (v_i,v_i))\in\Delta_{\alpha,\beta}$),
but not to $\eta_1$ (since $r\neq s$).
Therefore,
$((r,p), (s,q))\in(\eta_2\cap\Delta_{\alpha,\beta})-0$,
establishing that
$\eta_2\cap\Delta_{\alpha,\beta}\neq 0$.
This shows that if (c) fails because $[\alpha, \beta]>0$,
then (b) fails because $\eta_2\cap\Delta_{\alpha,\beta}>0$.

Now assume that (c) fails because $[\beta,\alpha\cap \beta]>0$.
There is a $\beta,\alpha\cap\beta$-matrix
\[
\begin{bmatrix}
p&q\\r&s\\
\end{bmatrix}
=
\begin{bmatrix}
t(\wec{e},\wec{u})&  t(\wec{e},\wec{v})\\
t(\wec{f},\wec{u})&  t(\wec{f},\wec{v})\\
\end{bmatrix}
\]
where $(e_i,f_i)\in \beta$, $(u_j,v_j)\in\alpha\cap \beta$,
$p=q$, and $r\neq s$.
The pair
\[
((p,q), (r,s))
=(t((\wec{e},\wec{e}),(\wec{u},\wec{v})), 
t((\wec{f},\wec{f}),(\wec{u},\wec{v})))
\]
belongs to $\Delta_{\alpha,\beta}$
(since $((e_i,e_i),(f_i,f_i))\in \Delta_{\alpha,\beta}$).
The pair $(r,p)$ belongs to $\beta$
(since $(p,r)=(t(\wec{e},\wec{u}),t(\wec{f},\wec{u}))$ and
$(e_i,f_i)\in\beta$).
Hence
\[
(r,r)\;\stackrel{\Delta_{\alpha,\beta}}{\equiv}\;
(p,p)=(p,q)\;\stackrel{\Delta_{\alpha,\beta}}{\equiv}\;(r,s).
\]
Therefore, $((r,r), (r,s)) \in \eta_1\cap \Delta_{\alpha,\beta}$,
but $((r,r), (r,s))\notin \eta_2$ since $r\neq s$.
This shows that if (c) fails because $[\beta,\alpha\cap \beta]>0$,
then (b) fails because $\eta_1\cap \Delta_{\alpha,\beta} > 0$.

At this point we have shown that
(a) $\Rightarrow$ (b) $\Rightarrow$ (c).
Interchanging the roles of $\alpha$ and $\beta$ in these arguments
we deduce that (c) $\Rightarrow$ (d) $\Rightarrow$ (a).
This shows that all four properties are equivalent.
The commutator congruence quasi-identity of the theorem
statement follows from the equivalence of (a) and (c).
\end{proof}

\begin{thm} \label{dist_implies_comm}
  If $\mathcal V$ has a Taylor term, then the
  commutator operation in $\mathcal V$ is
  left distributive if and only if it is
  right distributive. If either distributivity
  condition holds, then the commutator is
  commutative in $\mathcal V$.
\end{thm}

\begin{proof}
  By Theorem~\ref{distributive_thm}~(1),
  if $\mathcal V$ has left distributive commutator,
  then it has commutative commutator,
  hence it has right distributive commutator.
  This part of the theorem did not
  require the assumption that $\mathcal V$ has a Taylor term.

  Now assume that $\mathcal V$ has a Taylor term
  and has right distributive commutator.
  We shall argue that the commutator
  operation in $\mathcal V$ is commutative,
  hence left distributive. 
  This is a proof by contradiction, so assume also 
  that
  some algebra in $\mathcal V$
  has noncommutative commutator.
  By Lemma~\ref{asymmetry}, we may assume that
  some $\m a\in{\mathcal V}$
  has congruences $\alpha$ and $\beta$
  such that $0=[\beta,\alpha]<[\alpha,\beta]$.
Theorem~\ref{refinement} implies that
$0<[\alpha,\alpha\cap \beta]$.
This means that for $x:=\alpha$ and $y:=\alpha\cap \beta$
we have $y\leq x$, $0<[x,y]$, and
$[y,x]=[\alpha\cap\beta,\alpha]\leq [\beta,\alpha]=0$,
so $[x,y]\not\leq [y,x]$, in 
contradiction to Theorem~\ref{distributive_thm}~(2).
\end{proof}  

We strengthen the previous theorem with the following
result, which is one of the primary results of this article.

\begin{thm} \label{main2}
If $\mathcal V$ has a Taylor term, 
then the following are equivalent:
\begin{enumerate}
\item  $\mathcal V$ is congruence modular.
\item  The commutator is left distributive in $\mathcal V$.
\item  The commutator is right distributive in $\mathcal V$.
\end{enumerate}  
\end{thm}

\begin{proof}
For any congruence modular variety,
the commutator is both left and right distributive.
(See \cite[Proposition~4.3]{freese-mckenzie} for the
fact that the modular commutator is right distributive
and commutative.)
This shows that Item~(1) implies both Item~(2) and Item~(3).

By Theorem~\ref{dist_implies_comm},
Items~(2) and (3) are equivalent in the presence of a Taylor term,
and they imply that the commutator is commutative
in $\mathcal V$.
From commutativity, we derive the existence of a difference
term from Theorem~\ref{main1}.
Thus, it remains to prove that if $\mathcal V$
has (left) distributive commutator
and a difference term, then $\mathcal V$
is congruence modular.
This fact follows from Theorem~3.2~(i) of \cite{lipparini},
but we give an argument for this based on the results
of this paper.

We are going to argue by contradiction, 
so assume that the commutator is left distributive
throughout $\mathcal V$, but there is some
algebra in $\mathcal V$ that
does not have a modular congruence lattice.
We can find such an algebra $\m a\in \mathcal V$
with congruences
$\beta, \theta, \delta \in\Con(\m a)$
generating a pentagon satisfying $\delta < \theta$,
$\beta\cap \theta =0 \leq \delta$, and $\alpha:=\beta+\delta\geq \theta$
(Figure~\ref{fig12}).
\begin{figure}[ht]
\begin{center}
\setlength{\unitlength}{1mm}
\begin{picture}(20,33)
\put(0,15){\circle*{1.2}}
\put(10,0){\circle*{1.2}}
\put(10,30){\circle*{1.2}}
\put(20,10){\circle*{1.2}}
\put(20,20){\circle*{1.2}}

\put(10,0){\line(-2,3){10}}
\put(10,0){\line(1,1){10}}
\put(10,30){\line(-2,-3){10}}
\put(10,30){\line(1,-1){10}}
\put(20,20){\line(0,-1){10}}

\put(-4.5,13){$\beta$}
\put(22,9){$\delta$}
\put(22,19){$\theta$}
\put(9,32){$\alpha$}
\put(9,-5){$0$}
\end{picture}
\bigskip

\caption{\sc $\textrm{Con}(\m A)$.}\label{fig12}
\end{center}
\end{figure}
By the left distributivity of the commutator,
\[
[\alpha,\theta]=[\beta+\delta,\theta]=[\beta,\theta]+[\delta,\theta]
=0+[\delta,\theta] \leq \delta.  
\]
In this line we are using that $[\beta,\theta]\leq \beta\cap \theta=0$
to obtain the last equality and last inequality.

Let $\sigma$ denote $[\alpha,\theta]=[\delta,\theta]$,
a congruence which satisfies $0\leq \sigma\leq \delta$.
Working modulo $\sigma$, and writing $\overline{x}$
for $x/\sigma$ for any congruence $x\geq \sigma$, we have
\begin{enumerate}
\item $\overline{\alpha}\geq \overline{\theta}\geq \overline{\delta}$, and
\item $[\overline{\alpha},\overline{\theta}]=0$.
\end{enumerate}
Since $\mathcal V$ has a difference term,
Theorem~\ref{stable} guarantees that 
$\C C(\overline{\theta},\overline{\alpha};\overline{\delta})$ holds
in $\m a/\sigma$.
Since $\overline{\alpha}\geq \overline{\theta}\geq \overline{\delta}$, 
Theorem~\ref{basic_centrality}~(10) guarantees that
in $\m a/\delta$ we have
$\C C(\overline{\overline{\theta}},\overline{\overline{\alpha}};0)$,
where $\overline{\overline{x}}$ denotes $x/\delta$.
Because $\C C(\overline{\overline{\theta}},\overline{\overline{\alpha}};0)$
holds and the commutator is commutative in $\mathcal V$
we derive that
\[
0 = [\overline{\overline{\theta}},\overline{\overline{\alpha}}]=
[\overline{\overline{\alpha}},  \overline{\overline{\theta}}].
\]
Hence
$\C C(\overline{\overline{\alpha}},\overline{\overline{\theta}};0)$
holds in $\m a/\delta$.
Theorem~\ref{basic_centrality}~(10) guarantees that
$\C C(\alpha,\theta;\delta)$
holds in $\m a$, and so by monotonicity
$\C C(\beta,\theta;\delta)$
holds in $\m a$. This instance of centrality
in a pentagon contradicts
Theorem~\ref{memoir_pentagon}.
\end{proof}
\bigskip

Next we consider the existence of right annihilators
and the right semidistributivity
of the commutator.

\begin{thm} \label{main1.5}
If $\mathcal V$ has a Taylor term, 
then the following are equivalent:
\begin{enumerate}
\item  $\mathcal V$ has a difference term.
\item Right annihilators exist throughout $\mathcal V$.
\item The commutator is right semidistributive throughout $\mathcal V$.
\end{enumerate}  
\end{thm}

\begin{proof}
We shall argue that
$(1)\Rightarrow(2)\Rightarrow(3)$
and $\neg(1)\Rightarrow \neg(3)$.

Assume (1), that $\mathcal V$ has a difference term.
According to Theorem~\ref{main1},
$\mathcal V$ has commutative commutator.
Since the left annihilator of any congruence
on any algebra exists, 
right annihilators will also exist in any variety with a commutative
commutator (and $(0:\beta)_R = (0:\beta)_L$ will hold).
This proves that (2) holds.

Now assume that (2) holds.
Assume that $\m a\in {\mathcal V}$ and that
$\alpha,\beta,\beta'\in\Con(\m a)$
satisfy $[\alpha,\beta]=[\alpha,\beta']$.
The congruence $\delta:=[\alpha,\beta]$ is below
each of 
$\alpha, \beta, \beta'$ according to
Theorem~\ref{basic_centrality}~(7),
and both
$\C C(\alpha,\beta;\delta)$ and
$\C C(\alpha,\beta';\delta)$ hold since
$[\alpha,\beta]=\delta=[\alpha,\beta']$.
Let's factor by $\delta=[\alpha,\beta]$ to obtain
$\overline{\m a} = \m a/\delta\in{\mathcal V}$,
$\overline{\alpha}=\alpha/\delta$,
$\overline{\beta}=\beta/\delta$,
$\overline{\beta}'=\beta'/\delta$.
Both
$\C C(\overline{\alpha},\overline{\beta};0)$ and
$\C C(\overline{\alpha},\overline{\beta}';0)$ hold
in $\overline{\m a}$ by Theorem~\ref{basic_centrality}~(10).
This  implies that
$\overline{\beta}, \overline{\beta}'\leq (0:\overline{\alpha})_R$
in $\Con(\overline{\m a})$.
Hence
$\overline{\beta} + \overline{\beta}'\leq (0:\overline{\alpha})_R$.
Hence
$\C C(\overline{\alpha},\overline{\beta}+\overline{\beta}';0)$
in $\overline{\m a}$ by the definition of $(0:\overline{\alpha})_R$.
Hence
$\C C(\alpha,\beta+\beta';\delta)$ in $\m a$ by
Theorem~\ref{basic_centrality}~(10).
Hence
\begin{equation}\label{inequalities}
[\alpha,\beta]\leq [\alpha,\beta+\beta']\leq \delta=[\alpha,\beta].
\end{equation}
Here the first $\leq$ is an instance of the monotonicity of the commutator
in its second variable,
while the second $\leq$ follows from $\C C(\alpha,\beta+\beta';\delta)$
and the definition
of the commutator. Altogether, the line
(\ref{inequalities}) yields that 
$[\alpha,\beta] = [\alpha,\beta+\beta']$, completing the proof of the
right semidistributivity of the commutator.

The rest of the argument is devoted to establishing
the difficult implication $\neg(1)\Rightarrow \neg(3)$.
We start with the assumptions that $\mathcal V$ has
a Taylor term but does not have a difference term
and construct a failure of right semidistributivity
in the congruence lattice of some algebra in $\mathcal V$.
Since $\mathcal V$ has
a Taylor term but does not have a difference term,
Lemma~\ref{commutative_plus_Taylor_lm}
guarantees the existence of an algebra $\m a\in \mathcal V$
with a pentagon in its congruence lattice,
labeled as in Figure~\ref{fig11},
satisfying the commutator conditions (1), (2), and (3)
of that lemma. We shall start our construction
with the quotient $\m b=\m a/\delta$.
Writing $\overline{x}$ for $x/\delta$, 
Lemma~\ref{commutative_plus_Taylor_lm}
guarantees that $\m b$ has congruences
$0<\overline{\theta}<\overline{\alpha}$ such that
(i) $[\overline{\theta},\overline{\alpha}]=0$
(from Item (2) of that lemma) and
(ii) $[\overline{\alpha},\overline{x}]=\overline{x}$
for any congruence $\overline{x}$ satisfying
$0 \leq \overline{x}\leq\overline{\theta}$
(from Item (3) of that lemma).
In particular, we shall use the part of (ii) that guarantees that
$[\overline{\alpha},\overline{\theta}]=\overline{\theta}$.
Let $\m d = \m b(\overline{\alpha})$.
Write $\Delta$ for $\Delta_{\overline{\alpha},\overline{\theta}}\;(\in\Con(\m d))$.
Write $\Gamma$ for $\overline{\theta}_1\times \eta_2\;(\in\Con(\m d))$.
It is clear that $\Delta, \Gamma,
\leq \overline{\theta}_1\times \overline{\theta}_2$.
As usual, the coordinate projection kernels
on $\m d = \m b(\overline{\alpha})$ will be denoted
$\eta_1$ and $\eta_2$,
but to minimize subscripts below we shall 
use $\eta$ as a duplicate name for $\eta_1$ (that is, $\eta:=\eta_1$).
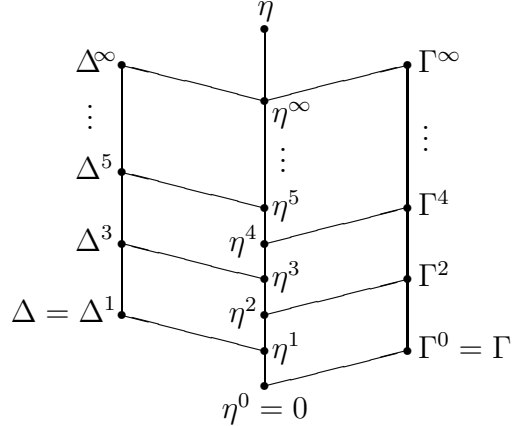
\begin{figure}[ht]
\begin{center}
\setlength{\unitlength}{.95mm}
\begin{picture}(40,50)
\put(20,0){\circle*{1.2}}
\put(20,5){\circle*{1.2}}
\put(20,10){\circle*{1.2}}
\put(20,15){\circle*{1.2}}
\put(20,20){\circle*{1.2}}
\put(20,25){\circle*{1.2}}
\put(20,40){\circle*{1.2}}
\put(20,50){\circle*{1.2}}

\put(40,5){\circle*{1.2}}
\put(40,15){\circle*{1.2}}
\put(40,25){\circle*{1.2}}
\put(40,45){\circle*{1.2}}

\put(0,10){\circle*{1.2}}
\put(0,20){\circle*{1.2}}
\put(0,30){\circle*{1.2}}
\put(0,45){\circle*{1.2}}

\put(0,10){\line(0,1){35}}
\put(20,0){\line(0,1){50}}
\put(40,5){\line(0,1){40}}

\put(20,0){\line(4,1){20}}
\put(20,10){\line(4,1){20}}
\put(20,20){\line(4,1){20}}

\put(20,5){\line(-4,1){20}}
\put(20,15){\line(-4,1){20}}
\put(20,25){\line(-4,1){20}}

\put(20,40){\line(4,1){20}}
\put(20,40){\line(-4,1){20}}

\put(41.5,4){$\Gamma^0=\Gamma$}
\put(-15.5,9){$\Delta=\Delta^1$}
\put(41.5,14){$\Gamma^2$}
\put(-6.5,19){$\Delta^3$}
\put(41.5,24){$\Gamma^4$}
\put(-6.5,29){$\Delta^5$}
\put(41.5,44){$\Gamma^{\infty}$}
\put(-6.5,44){$\Delta^{{}\!\!\infty}$}

\put(21,4){$\eta^1$}
\put(15,9){$\eta^2$}
\put(21,14){$\eta^3$}
\put(15,19){$\eta^4$}
\put(21,24){$\eta^5$}

\put(21,37){$\eta^{\infty}$}

\put(-5,36){$\vdots$}
\put(42,33){$\vdots$}
\put(22,30){$\vdots$}

\put(19,52){$\eta$}
\put(14,-4.5){$\eta^0=0$}
\end{picture}
\bigskip

\caption{\sc A herringbone-like portion of $\textrm{Con}(\m d)$.}\label{fig13}
\end{center}
\end{figure}
Define sequences of congruences
\begin{align} \label{eqns}
\Delta^1&= \Delta,\quad\quad \Delta^{2n+1}=\Delta+\eta^{2n}\\
\Gamma^0&= \Gamma,\quad\quad \Gamma^{2n+2}=\Gamma+\eta^{2n+1} \notag\\
\eta^0&= 0,\quad\quad
\eta^{2n+1}=\eta\cap \Delta^{2n+1}, \quad\quad
\eta^{2n}=\eta\cap \Gamma^{2n}, \notag \\
\Delta^{\infty}&=
\bigcup \Delta^{2n+1}, \quad\quad
\Gamma^{\infty} = \bigcup \Gamma^{2n}, \quad\quad
\eta^{\infty} = \bigcup \eta^n. \notag
\end{align}
See Figure~\ref{fig13} for a depiction
of the ordering of these congruences in $\Con(\m d)$.
This figure need not be a sublattice of $\Con(\m d)$
and the congruences in the figure
need not be distinct,
but the indicated (non-strict) comparabilities hold
and the meet or join
of any element in the central chain
with any element in a side chain is depicted correctly,
as we explain in the next claim.

\begin{clm} \label{order}
\mbox{}
  
\begin{enumerate}
\item $\eta^0\leq \eta^1\leq \eta^2\leq \cdots\leq \eta^{\infty}\leq \eta\cap(\Delta+\Gamma)$.
\item $\Delta=\Delta^1 \leq \Delta^3\leq \cdots \leq \Delta^{\infty}\leq \Delta+\Gamma$
and
  $\Gamma=\Gamma^0 \leq \Gamma^2\leq \cdots \leq \Gamma^{\infty}\leq \Delta+\Gamma$.
\item $\eta\cap \Delta^{\infty} = \eta^{\infty} = \eta\cap \Gamma^{\infty}$.  
\item $\Delta^{\infty}+\Gamma^{\infty}=\Delta+\Gamma\leq \overline{\theta}_1\times \overline{\theta}_2$.
\item The sets 
  $\{\eta^{2m+1}, \eta^{2m+2}, \eta^{2m+3}, \Delta^{2m+1}, \Delta^{2m+3}\}$ and
  $\{\eta^{2n}, \eta^{2n+1}, \eta^{2n+2}, \Gamma^{2n}, \Gamma^{2n+2}\}$
are sublattices of $\Con(\m d)$ for every $m, n\geq 0$.
\item The meet or join
of any element in the central chain
with any element in a side chain is depicted correctly
in Figure~\ref{fig13}.
\end{enumerate}
\end{clm}

{\it Proof of Claim~\ref{order}.}
For Item (1), observe that $\eta^0 = 0=\eta\cap\Gamma^0\leq \eta$.
If $\eta^{2n}\leq \eta$ for some $n$, then
$\eta^{2n} \leq \eta\cap(\Delta+\eta^{2n})=\eta^{2n+1}$
and $\eta^{2n+1}=\eta\cap(\Delta+\eta^{2n})\leq \eta$.
Similarly,
if $\eta^{2n+1}\leq \eta$ for some $n$, then
$\eta^{2n+1} \leq \eta\cap(\Gamma+\eta^{2n+1})=\eta^{2n+2}$
and $\eta^{2n+2}=\eta\cap(\Gamma+\eta^{2n+1})\leq \eta$.
Inductively we get that $\eta^0\leq \eta^1\leq \eta^2\leq\cdots$
and that all of these elements lie below $\eta$.

To complete the proof of (1) it will suffice
to show that $\eta^k\leq \Delta+\Gamma$ for all $k$.
This is true of $k=0$ since that $\eta^0=0$.
If $\eta^{2n}\leq \Delta+\Gamma$, then
$\Delta+\eta^{2n}\leq \Delta+\Gamma$, so 
$\eta^{2n+1}=\eta\cap(\Delta+\eta^{2n})\leq \Delta+\Gamma$.
Similarly, if $\eta^{2n+1}\leq \Delta+\Gamma$, then
$\Gamma+\eta^{2n+1}\leq \Delta+\Gamma$, so 
$\eta^{2n+2}=\eta\cap(\Gamma+\eta^{2n+1})\leq \Delta+\Gamma$.
By induction, $\eta^k\leq \Delta+\Gamma$ for all $k$
(hence $\eta^{\infty}\leq \Delta+\Gamma$, too).

For Item (2), the facts that (i) the $\eta$-sequence is increasing
and bounded above by $\Delta+\Gamma$ and (ii)
$\Delta^{2n+1}=\Delta+\eta^{2n}$ are jointly sufficient to imply
that the $\Delta$-sequence is increasing and bounded
above by $\Delta+\Gamma$.
A similar argument proves that the $\Gamma$-sequence
is increasing and bounded above by $\Delta+\Gamma$.

For Item (3),
$\eta\cap\Delta^{\infty} = \eta\cap\left(\bigcup \Delta^{2n+1}\right)
=\bigcup (\eta\cap \Delta^{2n+1}) 
=\bigcup \eta^{2n+1} 
=\eta^{\infty}.
$
Also
$\eta\cap\Gamma^{\infty} = \eta\cap\left(\bigcup \Gamma^{2n}\right)
=\bigcup (\eta\cap \Gamma^{2n}) 
=\bigcup \eta^{2n} 
=\eta^{\infty}.
$

For the equality in Item~(4), we have
$\Delta \leq \Delta^{\infty}\leq \Delta+\Gamma$
and
$\Gamma \leq \Gamma^{\infty}\leq \Delta+\Gamma$.
Joining these yields 
$\Delta+\Gamma \leq \Delta^{\infty}+\Gamma^{\infty}\leq \Delta+\Gamma$,
so $\Delta^{\infty}+\Gamma^{\infty}= \Delta+\Gamma$.
For the inequality in Item~(4), we have
$\Delta = \Delta_{\overline{\alpha},\overline{\theta}}\leq
\overline{\theta}_1\times \overline{\theta}_2$
and
$\Gamma = 
\overline{\theta}_1\times \eta_2\leq 
\overline{\theta}_1\times \overline{\theta}_2$,
so $\Delta+\Gamma \leq \overline{\theta}_1\times \overline{\theta}_2$.

We have already established in Items~(1) and (2)
that
$\eta^{2m+1}\leq \eta^{2m+2}\leq \eta^{2m+3}$
and $\Delta^{2m+1}\leq \Delta^{2m+3}$.
For the first part of Claim~\ref{order}~(5), it remains
to show that 
(i) $\Delta^{2m+1}+\eta^{2m+2}=\Delta^{2m+3}$
and (ii) $\Delta^{2m+1}\cap \eta^{2m+3}=\eta^{2m+1}$.
For (i), $\Delta^{2m+1}+\eta^{2m+2}=(\Delta+\eta^{2m})+\eta^{2m+2}=
\Delta+(\eta^{2m}+\eta^{2m+2})=\Delta+\eta^{2m+2}=\Delta^{2m+3}$.
For (ii), recall that 
$\eta^{2m+1}=\eta\cap \Delta^{2m+1}\leq \Delta^{2m+1}$.
Intersect the inequalities
$\eta^{2m+1}\leq \eta^{2m+3}\leq \eta$ throughout with 
$\Delta^{2m+1}$ to obtain 
$\eta^{2m+1}=\Delta^{2m+1}\cap \eta^{2m+1}\leq \Delta^{2m+1}\cap \eta^{2m+3}\leq \Delta^{2m+1}\cap \eta=\eta^{2m+1}$.
The middle value must equal the outer value, so 
$\Delta^{2m+1}\cap \eta^{2m+3}=\eta^{2m+1}$.

Item~(6) is a consequence of Items~(1), (2), and (5).
For example, the fact that $\eta^8+\Gamma^0=\Gamma^8$
may be argued:
\[
\begin{array}{rll}
\eta^8+\Gamma^0
&=(\eta^8+\eta^2)+\Gamma^0& (1)\\
&=\eta^8+(\eta^2+\Gamma^0)&\\
&=\eta^8+\Gamma^2& (5)\\
&=\eta^8+\Gamma^4& (1)+(5)\\
&=\eta^8+\Gamma^6& (1)+(5)\\
&=\Gamma^8& (5)
\end{array}
\]
while the fact that $\eta^8\cap \Gamma^0=0$
may be argued:
\[
\begin{array}{rll}
\eta^8\cap \Gamma^0
&=\eta^8\cap (\Gamma^6\cap \Gamma^0)& (2)\\
&=(\eta^8\cap \Gamma^6)\cap \Gamma^0&\\
&=\eta^6\cap \Gamma^0& (5)\\
&=\eta^4+\Gamma^0& (2)+(5)\\
&=\eta^2+\Gamma^0& (2)+(5)\\
&=0& (5)
\end{array}
\]
\cqed
\bigskip

From Claim~\ref{order}~(4) we have
$\eta\cap \Delta^{\infty} = \eta^{\infty}= \eta\cap \Gamma^{\infty}$.
It follows from this and Theorem~\ref{basic_centrality}~(8)
that $\C C(\eta,\Delta^{\infty};\eta^{\infty})$
and $\C C(\eta,\Gamma^{\infty};\eta^{\infty})$ hold.
The rest of the proof is devoted
to proving that
$\C C(\eta,\Delta^{\infty}+\Gamma^{\infty};\eta^{\infty})$
does \emph{not} hold. If we do this, then,
factoring by $\eta^{\infty}$
(which is below all congruences involved),
we get that in $\m d/\eta^{\infty}$ we have
the following failure of the right semidistributive law:
\[
[\eta/\eta^{\infty},\Delta^{\infty}/\eta^{\infty}]=0=
[\eta/\eta^{\infty},\Gamma^{\infty}/\eta^{\infty}],\;
\textrm{ but }
[\eta/\eta^{\infty},\Delta^{\infty}/\eta^{\infty}+\Gamma^{\infty}/\eta^{\infty}]\neq 0.
\]
Since $\Delta^{\infty}+\Gamma^{\infty}=\Delta+\Gamma$,
we can write our remaining goal as:

\begin{goal} \label{goal}
Show that $\C C(\eta,\Delta+\Gamma;\eta^{\infty})$ fails.
\end{goal}

Recall from the fourth paragraph of this proof
(i.e. of the proof of Theorem~\ref{main1.5})
that $0<\overline{\theta}<\overline{\alpha}$ and
$[\overline{\alpha},\overline{\theta}]=\overline{\theta}$.
This puts us in a position to mimic
the construction in Theorem~\ref{distributive_thm}~(2).
As was the case there
(with $\alpha, \beta$ there replaced by
$\overline{\alpha},\overline{\theta}$ here), 
there is an $\overline{\alpha},\overline{\theta}$-matrix
\[
\begin{bmatrix}
t(\wec{a},\wec{u}) &   t(\wec{a},\wec{v}) \\
t(\wec{b},\wec{u}) &   t(\wec{b},\wec{v}) 
\end{bmatrix}=
\begin{bmatrix}
p &  q \\
r &  s
\end{bmatrix}, \;\;\;
\wec{a}\wrel{\overline{\alpha}}\wec{b},\;\; \wec{u}\wrel{\overline{\theta}}\wec{v},
\]
with $p=q$ but $r\neq s$.
The fact that this is an
$\overline{\alpha},\overline{\theta}$-matrix
implies, in particular, that
$(p,r), (q,s)\in \overline{\alpha}$ and
$(p,q), (r,s)\in \overline{\theta}$.

\begin{clm} \label{special_matrix}
\begin{equation}
  \tag{MM}\label{MMM}
\begin{bmatrix}
t\left((\wec{b},\wec{a}),(\wec{u}, \wec{u})\right)&
t\left((\wec{b},\wec{a}),(\wec{u}, \wec{v})\right)\\
t\left((\wec{b},\wec{b}),(\wec{u}, \wec{u})\right)&
t\left((\wec{b},\wec{b}), (\wec{u}, \wec{v})\right)
\end{bmatrix}=
\begin{bmatrix}
(r, p) &  (r, q) \\
(r, r) &  (r, s)
\end{bmatrix}
\end{equation}
is an $\eta, (\Delta+\Gamma)$-matrix
of $\m d = \m b(\overline{\alpha})$
that is constant on the first row and not constant
on the second row.
\end{clm}

{\it Proof of Claim~\ref{special_matrix}.}
To show that the matrix given is truly an
$\eta, (\Delta+\Gamma)$-matrix,
we first argue that the elements
of the form
$(b_i,a_i)$, 
$(u_i,u_i)$,
$(b_i,b_i)$, and 
$(u_i,v_i)$ belong to the algebra
$\m d=\m b(\overline{\alpha})$.
This is so, because
$\wec{a}\wrel{\overline{\alpha}}\wec{b},\;\;
\wec{u}\wrel{\overline{\theta}}\wec{v}$,
$\overline{\theta}\subseteq \overline{\alpha}$,
and the universe of $\m d$
is $\overline{\alpha}$.

Next, we need to argue that
$((b_i,a_i),(b_i,b_i))\in\eta=\eta_1$,
and 
$((u_i,u_i), (u_i,v_i))\in\Delta+\Gamma$.
The former is clear, since 
$(b_i,a_i)$ and $(b_i,b_i)$
have the same first coordinate. The latter is clear, 
since $\Delta = \Delta_{\overline{\alpha},\overline{\theta}}$
and $\Gamma=\overline{\theta}_1\times\eta_2$
and for each subscript $i$ we have
$(u_i,v_i)\in\overline{\theta}$, so
\[
(u_i,u_i)\stackrel{\Delta}{\equiv}
(v_i,v_i) \stackrel{\Gamma}{\equiv}
(u_i,v_i).
\]

We have shown that the matrix in \eqref{MMM} is truly
an $\eta_1, (\Delta+\Gamma)$-matrix.
The first row is constant and the second is not
(since $p=q$ and $r\neq s$ as one sees in the lines
before the statement of Claim~\ref{special_matrix}).
\cqed
\bigskip

Establishing Goal~\ref{goal}, that 
$\C C(\eta,\Delta+\Gamma;\eta^{\infty})$ fails,
is equivalent to establishing that there exists
some $\eta,(\Delta+\Gamma)$-matrix whose first row
lies in $\eta^{\infty}$ and whose second
row does not. We shall argue that 
the matrix in \eqref{MMM} is such a matrix.
Already we know from Claim~\ref{special_matrix}
that this matrix is an 
$\eta,(\Delta+\Gamma)$-matrix. We also know
that the first row
lies in $\eta^{\infty}$, since the first row
is constant. The rest of the proof
is devoted to showing that the second
row does not belong to
$\eta^{\infty}$. 
For this, let
$V=r/\overline{\theta}$ be the $\overline{\theta}$-class
of $r$ in $\m b$ and let 
$U=V\times V =
(r, r)/(\overline{\theta}_1\times\overline{\theta}_2)$ be the
$\overline{\theta}_1\times\overline{\theta}_2$-class
of $(r, r)$ in $\m b(\overline{\alpha})$
and let $0_U$ be the equality relation on 
$U$. Since $(r,s)\in\overline{\theta}$,
we have $(r,r), (r,s)\in U$.
Since $r\neq s$, we also have
$(r,r) \neq (r,s)$.
We shall accomplish Goal~\ref{goal} by showing that
$\eta^{\infty}|_U = 0_U$, so
$((r,r),(r,s))\notin \eta^{\infty}$.

\begin{clm} \label{U}
  $U$ is a union of $\Delta$-classes
  and a union of $\Gamma$-classes.
In fact, $U$ is a union of congruence classes for each of the congruences
  $\Delta^{2n+1}, \Gamma^{2n}, \eta^k$ and
  $\Delta^{\infty}, \Gamma^{\infty}, \eta^{\infty}$.
\end{clm}  

{\it Proof of Claim~\ref{U}.}
Since $U$ is a single class
of the congruence $\overline{\theta}_1\times\overline{\theta}_2$,
it is a union of congruence classes of any smaller
congruence. According to
Claim~\ref{order}, all of the congruences
$\Delta=\Delta^1$, $\Gamma=\Gamma^0$, 
  $\Delta^{2n+1}$, $\Gamma^{2n}$, $\eta^k$,
$\Delta^{\infty}$, $\Gamma^{\infty}$, $\eta^{\infty}$
are contained in $\Delta+\Gamma$, which is
contained in 
$\overline{\theta}_1\times\overline{\theta}_2$.
\cqed
\bigskip

\begin{clm} \label{restriction_is_zero1}
$\eta^1|_U= 0_U$.
\end{clm}  

{\it Proof of Claim~\ref{restriction_is_zero1}.}
This claim is proved by localizing the proof
of (a)$\Rightarrow$(b) of Theorem~\ref{refinement}
to the congruence class $U$.

Since $0<\overline{\theta}<\overline{\alpha}$
in $\Con(\m b)$ 
and $0=[\overline{\theta},\overline{\alpha}]\;(\geq [\overline{\theta},\overline{\theta}])$
we get that $\overline{\theta}$ is an abelian
congruence of $\m b$ and hence
$\overline{\theta}_1\times \overline{\theta}_2$ is an abelian
congruence of $\m b(\overline{\alpha})$.
The $\overline{\theta}_1\times \overline{\theta}_2$-class
$U$ is therefore a class of an abelian congruence.

Recall that $\eta = \eta_1$ is the first projection
kernel of $\m d = \m b(\overline{\alpha})$.
If 
$((c,a),(c,b))\in\eta^1|_U = \eta|_U\cap \Delta|_U$,
then since $(c,a), (c,b)\in U$ it must be that
$a\stackrel{\overline{\theta}}{\equiv}c\stackrel{\overline{\theta}}{\equiv}b$.
Let $T(x_1,\ldots,x_{n})$ be a Taylor term for $\mathcal V$.
Consider a first-place Taylor identity
$T(x,\wec{w})\approx T(y,\wec{z})$ where
$\wec{w}, \wec{z}\in \{x,y\}^{n-1}$.
Substitute $b$ for all occurrences of $x$ and $c$
for all occurrences of $y$.
This yields $T(b,\bar{u}) = T(c,\bar{v})$
where all $u_i$ and $v_i$ are in $\{b,c\}$.
The facts that $\overline{\theta}\leq \overline{\alpha}$,
$[\overline{\theta},\overline{\alpha}]=0$, and
$U$ is a $\overline{\theta}_1\times \overline{\theta}_2$-class
imply that the diagonal of $U$ is a single
$\Delta=\Delta_{\overline{\alpha},\overline{\theta}}$-class of
$\m d = \m b(\overline{\alpha})$.
As in the proof of
Theorem~\ref{refinement},
the fact that $(T(b,\wec{u}),T(c,\wec{v}))$
lies on the diagonal of $U$ implies that
$(T(a,\wec{u}),T(c,\wec{v}))=(T(b,\wec{u}),T(c,\wec{v}))$,
so $T(a,\wec{u})=T(b,\wec{u})$
where each $u_i\in \{b,c\}$.
By the $\overline{\theta},\overline{\theta}$-term
condition, 
$T(a,\bar{y}) = T(b,\bar{y})$
for any $\bar{y}$ whose entries
lie in the $\overline{\theta}$-class containing $a, b$ and $c$.
As in the proof of
Theorem~\ref{refinement},
this conclusion holds in every place of $T$.
That is,
\[T(y_1,\ldots,y_{i-1},a,y_{i+1},\ldots,y_n) = 
T(y_1,\ldots,y_{i-1},b,y_{i+1},\ldots,y_n)\]
for each $i$ and any choice of values for 
$y_1,\ldots,y_n$ in the $\overline{\theta}$-class of $a, b$, and $c$.
Using the
fact that $T$ is idempotent, we have
\[
a=T(a,a,\ldots,a)=T(b,a,\ldots,a)=\cdots=T(b,b,\ldots,b)=b.
\]
This proves that $(c,a)=(c,b)$.
Since $((c,a),(c,b))\in\eta^1|_U$
was arbitrarily chosen,
$\eta^1|_U = \eta|_U\cap \Delta|_U = 0_U$.
\cqed
\bigskip

\begin{clm} \label{restriction_is_zero}
$\eta^{\infty}|_U = 0_U$.
\end{clm}  

{\it Proof of Claim~\ref{restriction_is_zero}.}
Since $U$ is a single 
$\overline{\theta}_1\times\overline{\theta}_2$-class,
the restriction map from the congruence interval
$I[0,\overline{\theta}_1\times\overline{\theta}_2]$
in $\Con(\m d)$ to the lattice of equivalence relations
on the set $U$ is a complete lattice homomorphism.
(Under this map a congruence $x$ maps
to $x|_U=x\cap (U\times U)$.)
If you apply this restriction map to the congruences in
Figure~\ref{fig13} (excluding $\eta$, which need not
be in the interval $I[0,\overline{\theta}_1\times\overline{\theta}_2]$)
you will get a similarly-ordered
set of equivalence relations on $U$.
If you replace each congruence $x$
in Figure~\ref{fig13} with $x|_U$,
then all the claims of Claim~\ref{order}
remain true.

In particular, the set
$\{\eta^{0}|_U, \eta^{1}|_U, \eta^{2}|_U, \Gamma^0|_U, \Gamma^2|_U\}$
is a sublattice that is a quotient of a pentagon.
By Claim~\ref{restriction_is_zero1},
$\eta^0|_U = 0_U = \eta^1|_U$. Since we are dealing with a quotient
of a pentagon, we derive that 
$\Gamma^0|_U = \Gamma^2|_U$, since 
$\Gamma^0|_U = \Gamma^0|_U+\eta^0|_U=\Gamma^0|_U+\eta^1|_U=\Gamma^2|_U$.
Then we have $\eta^2|_U = \eta^1|_U = \eta^0|_U$, since
$\eta^2|_U = \eta^2|_U\cap \Gamma^2|_U=\eta^2|_U\cap \Gamma^0|_U=\eta^0|_U$.
In summary, from
$\eta^1|_U = \eta^0|_U$ we derive
$\Gamma^2|_U = \Gamma^0|_U$ from which we derive
$\eta^2|_U = \eta^1|_U$. A similar argument now allows us to derive
from $\eta^2|_U = \eta^1|_U$ that
$\Delta^3|_U = \Delta^1|_U$ from which we derive
$\eta^3|_U = \eta^2|_U$.
This may be continued to derive
$\eta^0|_U = \eta^1|_U = \eta^2|_U = \cdots$,
$\Gamma^0|_U = \Gamma^2|_U = \Gamma^4|_U = \cdots$, and
$\Delta^1|_U = \Delta^3|_U = \Delta^5|_U = \cdots$.
Taking the complete joins of these constant sequences
we get
$\eta^{\infty}|_U = \eta^0|_U = 0$,
$\Gamma^{\infty}|_U = \Gamma^0|_U$, and 
$\Delta^{\infty}|_U = \Delta^0|_U$.
\cqed
\bigskip

\begin{clm} \label{centrality_fails}
$\C C(\eta,\Delta+\Gamma;\eta^{\infty})$ fails.
\end{clm}  

{\it Proof of Claim~\ref{centrality_fails}.}
From Claim~\ref{special_matrix} we know
that matrix \eqref{MMM} is an 
$\eta,(\Delta+\Gamma)$-matrix whose
first row is constant, hence the elements in the first row
are congruent
modulo $\eta^{\infty}$.
The second row is not constant
but lies in $U$. Since $\eta^{\infty}|_U=0_U$,
the elements in the second row
are not congruent
modulo $\eta^{\infty}$. Thus matrix \eqref{MMM}
witnesses the failure of $\C C(\eta,\Delta+\Gamma;\eta^{\infty})$.
\cqed
\bigskip

We complete the proof of Theorem~\ref{main1.5}
by reiterating ideas mentioned before the statement
of Goal~\ref{goal}.
By Claim~\ref{order}~(3),
$\eta\cap \Delta^{\infty}=\eta\cap \Gamma^{\infty}=\eta^{\infty}$,
so all of the congruences
$\eta, \Delta^{\infty}, \Gamma^{\infty}$ lie above $\eta^{\infty}$.
In $\Con(\m d/\eta^{\infty})$, let
$x = \eta/\eta^{\infty}, 
y = \Delta^{\infty}/\eta^{\infty}, 
z = \Gamma^{\infty}/\eta^{\infty}$. We have
$x\cap y = 0 = x\cap z$, so $\C C(x,y;0)$
and $\C C(x,z;0)$ hold in $\m d/\eta^{\infty}$.
But we do not have
$\C C(x,y+z;0)$ in $\m d/\eta^{\infty}$, since this translates to
$\C C(\eta,\Delta^{\infty}+\Gamma^{\infty};\eta^{\infty})$
in $\m d$, which is the same statement as 
$\C C(\eta,\Delta+\Gamma;\eta^{\infty})$.
We proved in Claim~\ref{centrality_fails}
that $\C C(\eta,\Delta+\Gamma;\eta^{\infty})$ fails in $\m d$.
\end{proof}

Next
we show that four of the centralizer properties
from the Introduction are Maltsev definable relative
to the existence of a Taylor term.
The following is another of the primary results
of this article.

\begin{thm} \label{main3}
Let $\mathcal V$ be a variety that has a Taylor term.
The following are equivalent properties for $\mathcal V$:
\begin{enumerate}
\item[(1)]  $\mathcal V$ is congruence modular.
\item[(2)] The centralizer relation is
symmetric in its first two places throughout $\mathcal V$. \newline  
{\rm ($\C C(x,y;z)\Longleftrightarrow \C C(y,x;z)$.)}
\item[(3)] Relative right annihilators exist. \newline
  {\rm (Given $x, z$, there is a largest $y$
    such that $\C C(x,y;z)$. Write $y=(z:x)_R$.)}  
\item[(4)]
  The centralizer relation is determined by the commutator
  throughout $\mathcal V$. \newline
({\rm $\C C(x,y;z)\Longleftrightarrow [x,y]\leq z$.})
\item[(5)]
  The centralizer relation is
stable under lifting in its third place throughout $\mathcal V$. \newline
{\rm ($\C C(x,y;z)\;\&\; (z\leq z')\Longrightarrow \C C(x,y;z')$.)}
\end{enumerate}
\end{thm}

\begin{proof}
The fact that
the bi-implication in Item~(2) holds in every
congruence modular variety
is proved in \cite[Proposition~4.2]{freese-mckenzie}.
(The bi-implication is
(1)(iii)$\Leftrightarrow$(1)(iv) of Proposition~4.2 of the reference.)
Hence (1)$\Rightarrow$(2).
We next explain how to derive (3) from (2):
It follows from 
Theorem~\ref{basic_centrality}~(5) that the
relative \emph{left} annihilator, $(\delta:\theta)_L$, exists for any
$\delta,\theta\in\Con(\m a)$ on any algebra $\m a$.
From (2), which asserts the symmetry of the centralizer relation
in its first two places, it follows that relative
right annihilators must also exist and
that $(\delta:\theta)_R=(\delta:\theta)_L$
for every $\delta$ and $\theta$.

The fact that
the bi-implication in Item~(4) holds in every
congruence modular variety is
(1)(iii)$\Leftrightarrow$(1)(v) of \cite[Proposition~4.2]{freese-mckenzie}.
Hence (1)$\Rightarrow$(4).
We also have (4)$\Rightarrow$(5), for the following reason:
(4) asserts that the centralizer relation is
equivalent to the relation $[x,y]\leq z$, which
is a relation that is stable under lifting in $z$.
Thus, (4)
implies that the centralizer is stable under lifting in
its third place, which is Item (5).

So far we have (1)$\Rightarrow$(2)$\Rightarrow$(3)
and
(1)$\Rightarrow$(4)$\Rightarrow$(5).
To finish the proof it will suffice to establish
(3)$\Rightarrow$(1) and
(5)$\Rightarrow$(1).

In this paragraph we prove
(3)$\Rightarrow$(1) by contradiction.
Therefore, assume that (3) holds
(relative right annihilators exist)
and (1) fails ($\mathcal V$ is not congruence modular).
We also assume throughout the proof that
the global hypotheses of the theorem hold
($\mathcal V$ is a variety that has a Taylor term).
If relative right annihilators (those of the form $(\delta:\theta)_R$)
always exist,
then ordinary right annihilators
(those of the form $(0:\theta)_R$)
must also exist.
Since $\mathcal V$ has a Taylor term
and the property that ordinary right annihilators
exist throughout $\mathcal V$, it follows from
Theorem~\ref{main1.5} that $\mathcal V$ has a difference term.
Since we have assumed that
$\mathcal V$ is not congruence modular, 
there will exist an algebra $\m a\in {\mathcal V}$
with a pentagon in $\Con(\m a)$.
We label it as in Figure~\ref{fig4}.
For this algebra we have
$\C C(\theta,\beta;\delta)$ and
$\C C(\theta,\delta;\delta)$ by
Theorem~\ref{basic_centrality}~(7).
Hence $\beta, \delta\leq (\delta:\theta)_R$,
which forces $\alpha = \beta+\delta\leq (\delta:\theta)_R$,
or equivalently $\C C(\theta,\alpha;\delta)$.
By monotonicity in the middle place
we derive $\C C(\theta,\theta;\delta)$,
which implies that the critical interval
of the pentagon, $I[\delta,\theta]$,
is abelian. But according to Theorem~\ref{diff_char},
critical intervals of pentagons are neutral
in varieties with a difference term.
We have arrived at a contradiction, since
nontrivial congruence intervals like $I[\delta,\theta]$
cannot be both abelian ($[\theta,\theta]_{\delta}=\delta$)
and neutral ($[x,x]_{y}=x$ for $\delta\leq y\leq x\leq \theta$).
This completes the proof that (1), (2), and (3) are equivalent.

In this paragraph we prove
(5)$\Rightarrow$(1) by contradiction.
Therefore, assume that (5) holds
(the centralizer is stable under lifting
in its third place)
and (1) fails ($\mathcal V$ is not congruence modular).
We proceed as above:
since we have assumed that
$\mathcal V$ is not congruence modular, 
there will exist an algebra $\m a\in {\mathcal V}$
with a pentagon in $\Con(\m a)$, which 
we label as in Figure~\ref{fig4}.
According to
Theorem~\ref{basic_centrality}~(8),
we have that
$\C C(\beta,\theta;\beta\cap\theta)$ holds.
Since we have assumed that
the centralizer relation is stable under lifting
in its third place, $\C C(\beta,\theta;\delta)$ holds.
This contradicts Theorem~\ref{memoir_pentagon}.
This completes the argument that Items
(1), (4), and (5) are equivalent.
\end{proof}

Finally we show that the property of weak stability
of the centralizer relation in its third
place is Maltsev definable relative
to the existence of a Taylor term.
This is our last primary result.

\begin{thm} \label{main4}
Let $\mathcal V$ be a variety that has a Taylor term.
The following are equivalent properties for $\mathcal V$:
\begin{enumerate}
\item[(1)]  $\mathcal V$ has a difference term.
\item[(2)] The centralizer relation is weakly
stable under lifting in its third place throughout $\mathcal V$.
{\rm ($\C C(x,y;z)\;\&\; (z\leq z'\leq x\cap y)\Longrightarrow \C C(x,y;z')$.)}
\end{enumerate}
\end{thm}

\begin{proof}
  For (1)$\Rightarrow$(2), assume that $\mathcal V$
  has a difference term
  and that some $\m a\in {\mathcal V}$ has congruences
$x=\alpha, y=\beta, z=\delta, z'=\gamma$ such that
$\C C(\alpha,\beta;\delta)\;\&\; (\delta\leq \gamma\leq \alpha\cap \beta)$.
By \cite[Lemma~2.3~(i)$\Rightarrow$(ii)]{diff},
$\C C(\alpha,\beta;\delta)$ implies
$[\alpha,\beta]_{\delta}=\delta$.
It now follows from
\cite[Lemma~2.4]{diff} that $[\alpha,\beta]_{\gamma}=[\alpha,\beta]+\gamma
=\gamma$. By \cite[Lemma~2.3~(ii)$\Rightarrow$(i)]{diff},
$\C C(\alpha,\beta;\gamma)$ holds.
This establishes the weak stability property.

Now we argue that if $\mathcal V$ has a Taylor term and
does not have a difference term,
then the centralizer will not be weakly stable in its third
place in some instances.
By Lemma~\ref{commutative_plus_Taylor_lm},
$\mathcal V$ has an algebra $\m a$
with a pentagon in its congruence lattice,
which satisfies the commutator conditions (1), (2), and (3)
of that lemma.
Take $x=\alpha, y=\theta, z=0, z'=\delta$.
From the lemma, $[\alpha,\theta]=0$, so
$\C C(x,y;z)$ holds.
By our choices, $z\leq z'\leq x\cap y$.
Since $[\alpha,\theta]_{\delta}=\theta$ we have
$\neg \C C(x,y;z')$. This shows that weak stability fails.
\end{proof}

\section{Outro}
\subsection{The intended applications}

The results of this paper help
to decide some cases of the following question:
Given a finite algebra $\m a$ of finite type, does
the variety $\mathcal V = {\sf H}{\sf S}{\sf P}(\m a)$
have commutative commutator?
If $\mathcal V$ has a Taylor term, then the answer
is affirmative if and only if $\mathcal V$
also has a difference term.
There are known algorithms 
to decide whether $\mathcal V$ has a Taylor term and 
whether $\mathcal V$ has a difference term
whenever $\mathcal V$ is generated by a finite algebra of finite type.
(These algorithms are implemented in UACalc, \cite{uacalc}).
This gives a path to answer the question
algorithmically for finite algebras of finite type
that have a Taylor term.

Even in the case where $\mathcal V = {\sf H}{\sf S}{\sf P}(\m a)$
does not have a Taylor term,
the results of this paper might apply.
Suppose that $\m a$ is
a finite algebra of finite type and
$\mathcal V = {\sf H}{\sf S}{\sf P}(\m a)$
does not have a Taylor term.
It is possible that some
$\m b\in {\mathcal V}$ generates a subvariety
${\mathcal U} ={\sf H}{\sf S}{\sf P}(\m b)$
that has a Taylor
term but does not have a difference term.
In this case, the
subvariety
will not have commutative commutator, so
$\mathcal V$ cannot have commutative commutator.
If there is such a $\m b\in {\mathcal V}$,
then there must exist such a $\m b$ that is
free on three generators in the subvariety it generates,
hence will be a quotient of the finite, relatively free algebra
$\m f_{\mathcal V}(3)$. Determining whether
such a $\m b$ exists is a matter of a finite
amount of computation. A concrete example where
this happens is when $\m a$ is the semigroup
$\mathbb Z_2\times \mathbb S_2\times \mathbb L_2$
and
$\m b = \mathbb Z_2\times \mathbb S_2$. (Here
$\mathbb Z_2$ is the $2$-element group considered as a semigroup,
$\mathbb S_2$ is the $2$-element semilattice, and
$\mathbb L_2$ is the $2$-element left zero semigroup.)
In this example, 
${\mathcal V} ={\sf H}{\sf S}{\sf P}(\m a)$
does not have a Taylor term,
but one may still apply the results of this
paper to derive that the commutator
is not commutative in $\mathcal V$ since the
subvariety
${\mathcal U} ={\sf H}{\sf S}{\sf P}(\m b)$
has a Taylor term and does not have a difference term.
(Contrast with this example: the algebra
$\m c = \mathbb S_2\times \mathbb L_2$
generates a variety with noncommutative commutator,
but the results of this paper do not help
to establish this since
every subvariety of ${\sf H}{\sf S}{\sf P}(\m c)$
that has a Taylor term also has a difference term.)

\bigskip

Another intended application of the results of this
paper is to help understand whether some theorems are
expressed with optimal hypotheses.
For example, in \cite{park},
\'{A}gnes Szendrei, Ross Willard and I
proved Park's Conjecture for varieties
with a difference term.
Park's Conjecture is the conjecture that a finitely generated
variety of finite type is finitely based
whenever it has a finite residual bound.
One question received after the publication
of that paper
was: How hard would it be to generalize the proof
in \cite{park},
which assumes the existence of a difference term,
to establish Park's Conjecture for varieties with a
\emph{weak} difference term?\footnote{Note: %
A finitely generated variety has a weak difference term if and only
if it has a Taylor term.}
Our proof in \cite{park} depends on the commutativity
of the commutator in some places. Thus one may ask:
if one were to refine the proof in \cite{park}
so that it proves
Park's Conjecture for varieties
that have a \emph{weak} difference term and commutative commutator,
would this refinement constitute a proper generalization
of the result in \cite{park}?
The answer is negative, according to 
Theorem~\ref{commutative_plus_weak} of this paper.
That is, the class of varieties which have a weak
difference term and commutative commutator
is exactly the same as the class of varieties
with a difference term. Any proper generalization
of the result in \cite{park} must apply to
some varieties $\mathcal V$
in which either (i) $\mathcal V$
has no Taylor term or (ii) 
the commutator operation in $\mathcal V$
is not commutative.

\bigskip

\subsection{Some problems from \cite{lipparini}}

The results of this article partially solve
some problems posed by Paolo Lipparini in \cite{lipparini}.
The problems I refer to are:
\bigskip

\noindent
{\bf Problems 1.7 of \cite{lipparini}.}
    
\begin{enumerate}
\item[(a)]
Find conditions implying (if possible, equivalent to)
left join distributivity, right join distributivity or
commutativity of the commutator.
\bigskip
  
\item[(b)]
In particular, is there a (weak) Mal'cev condition strictly weaker than
modularity and implying left join distributivity of the commutator?  
\bigskip
  
\item[(c)]
Does right join distributivity always imply left join distributivity?
\bigskip
  
\item[(e)]
Answer the above questions at least in the
particular cases of varieties
with a (weak) difference term, $M$-permutable varieties,
locally finite varieties
(omitting type ${\bf 1}$ or some other type).
\end{enumerate}
\bigskip

\noindent
{\bf Partial solutions.}

\noindent
In this article we work at the level of varieties.
At this level, we can say the following.

Regarding Problem 1.7(a),
we have characterized those varieties with a Taylor term
that have left distributive, right distributive, or
commutative commutator.

Regarding Problem 1.7(b),
Theorem~\ref{main2} implies 
that there is \underline{no}
idempotent Maltsev condition strictly weaker than
modularity that implies left distributivity of the commutator.

Regarding Problem 1.7(c), 
Theorem~\ref{distributive_thm}~(1) shows
that left distributivity of the commutator
throughout a variety implies commutativity
of the commutator throughout the variety.
Hence left distributivity implies right distributivity
in any variety.
We do not know if, conversely, right distributivity
implies left distributivity in every variety.
Nevertheless we have shown that left and right
distributivity are equivalent for 
varieties with a Taylor term in Theorem~\ref{dist_implies_comm}.

Regarding Problem 1.7(e), if some variety $\mathcal V$
has a difference term, a weak difference term,
is $M$-permutable, or is
a locally finite omitting type ${\bf 1}$, then
$\mathcal V$ has a Taylor term. In these settings we have
classified the varieties that have
left distributive, right distributive, or
commutative commutator.

\bigskip
  
\subsection{Two problems from \cite{64problems}}
The results of this article solve two problems
from the list of 64 open problems posed at
the Workshop on Tame Congruence Theory
which was held at the 
Paul Erd\H{o}s
Summer Research Center of Mathematics in 2001.
\bigskip

\noindent
{\bf Problem 10.6 of \cite{64problems}.}
Let $\mathcal V$
be a locally finite variety that omits type {\bf 1}.
Is it true that if $[\alpha,\beta] = [\beta,\alpha]$
for all congruences $\alpha, \beta$
of algebras in $\mathcal V$, then $\mathcal V$ has a difference term?    
\bigskip

A locally finite variety omits type {\bf 1}
if and only if it has a Taylor term, according to 
Lemma~9.4 and 
Theorem~9.6 of \cite{hobby-mckenzie}.
Thus, Problem~10.6 of \cite{64problems}
asks about the truth of
Theorem~\ref{main1} of this paper
in the restricted setting of 
locally finite varieties.
Theorem~\ref{main1}
provides an affirmative answer.

\bigskip

\noindent
{\bf Problem 10.7 of \cite{64problems}.}
Are there natural conditions on a variety $\mathcal V$
under which the implications
$$
[\alpha,\beta]=[\alpha,\gamma] \Longrightarrow [\alpha,\beta]=[\alpha,\beta+\gamma]
\;\;\textrm{ and }\;\;
[\beta,\alpha]=[\gamma,\alpha] \Longrightarrow 
[\beta,\alpha]=[\beta+\gamma,\alpha]
$$
hold throughout the variety $\mathcal V$?
(Consider, e.g., the condition `$\mathcal V$ has a difference term'.)
\bigskip

Problem~10.7 of \cite{64problems}
asks for natural conditions guaranteeing
the right or left semidistributivity
of the commutator for varieties,
and suggests that having a difference term
might be such a condition.
Every variety has left semidistributive
commutator by Theorem~\ref{basic_centrality}~(5)
and the definition of the commutator,
so the nontrivial part of this problem is the question about
right semidistributivity.
Theorem~\ref{main1.5} proves that, for varieties
with a Taylor term,
the condition proposed in Problem 10.7 of \cite{64problems}
(that $\mathcal V$ has a difference term)
is a necessary and sufficient condition
guaranteeing that $\mathcal V$ has right (and left)
semidistributive commutator.

\bigskip

\subsection{A problem from \cite{novi_sad}}
The results of this article solve a problem
I posed at the 90th Arbeitstagung Allgemeine Algebra
held at the University of Novi Sad in 2015. There
I gave a talk entitled
\emph{Problems on the frontier of commutator theory}.
These twenty-five problems were not published formally,
but the slides for the talk are posted at \cite{novi_sad}.
The thirteenth problem asks
\bigskip

\noindent
{\bf Problem.} Does
$\exists$weak difference term +
symmetric commutator imply
$\exists$difference term?

\bigskip

\noindent
This problem is answered affirmatively in
Theorem~\ref{commutative_plus_weak} of this paper.
The affirmative answer is
strengthened in two ways to
\begin{center}
\underline{$\exists$Taylor term} +
symmetric commutator \underline{$\Longleftrightarrow$}
$\exists$difference term.
\end{center}
in Theorem~\ref{main1}.
The conclusion is: If you want to prove
some case of some conjecture about varieties,
and you need (i) a Taylor term and (ii) a commutative
commutator throughout your variety for the proof,
then the assumption that the variety has a difference term
guarantees both (i) and (ii) and it is the optimal
hypothesis that guarantees both.

\bibliographystyle{plain}

\end{document}